\documentclass[3p]{elsarticle}
\usepackage[utf8]{inputenc}
\usepackage{amssymb}
\usepackage{amsthm}
\usepackage{amsmath}
\usepackage{booktabs}
\usepackage{braket}
\usepackage{nicefrac}
\usepackage{subcaption}

\usepackage{siunitx}
\usepackage[table]{xcolor}

\usepackage{hyperref}

\numberwithin{equation}{section}

\newcommand{\norm}[2]{\left\lvert \left\lvert #1 \right\rvert \right\rvert_{#2}}
\newcommand{\abs}[1]{\left\lvert #1 \right\rvert}
\newcommand{\dx}[1]{\ \text{d}#1}
\newcommand{\laplace}{\Delta}
\newcommand{\grad}{\nabla}

\newcommand{\inner}[3]{\left( #1 , #2 \right)_{#3}}
\newcommand{\scalar}[2]{\left( #1, #2 \right)}
\newcommand{\dual}[3]{\left\langle #1, #2 \right\rangle_{#3}}

\newcommand{\embedding}{\hookrightarrow}

\newtheorem{Theorem}{Theorem}[section]
\newtheorem{Corollary}[Theorem]{Corollary}
\newdefinition{Remark}[Theorem]{Remark}
\newtheorem{Lemma}[Theorem]{Lemma}
\newtheorem{Assumption}[Theorem]{Assumption}
\journal{Journal of Mathematical Analysis and Applications}

\begin{document}

\begin{frontmatter}
	
	\title{Asymptotic Analysis for Optimal Control of the Cattaneo Model}
	
	\author[ITWM]{Sebastian Blauth}\corref{mycorref}
	\cortext[mycorref]{Corresponding author}
	\ead{sebastian.blauth@itwm.fraunhofer.de}
	
	\author[RPTU]{René Pinnau}
	\ead{pinnau@mathematik.uni-kl.de}
	
	\author[RPTU]{Matthias Andres}
%	\ead{andres@mathematik.uni-kl.de}
	
	\author[BUW]{Claudia Totzeck}
	\ead{totzeck@uni-wuppertal.de}
	
	\address[ITWM]{Fraunhofer Institute for Industrial Mathematics ITWM, Fraunhofer-Platz 1, D-67663 Kaiserslautern}
	\address[RPTU]{RPTU Kaiserslautern Landau, Gottlieb-Daimler-Str. 48, D-67663 Kaiserslautern}
	\address[BUW]{University of Wuppertal, Gau{\ss}straße 20, D-42119 Wuppertal}
	
	{
		This is a post-peer-review, pre-copyedit version of an article published in Journal of Mathematical Analysis and Applications. The final version is available online at \url{https://doi.org/10.1016/j.jmaa.2023.127375}.
	}
	
	\begin{abstract}
		We consider an optimal control problem with tracking-type cost functional constrained by the Cattaneo equation, which is a well-known model for delayed heat transfer. In particular, we are interested in the asymptotic behaviour of the optimal control problems for a vanishing delay time $\tau \rightarrow 0$. First, we show the convergence of solutions of the Cattaneo equation to the ones of the heat equation. Assuming the same right-hand side and compatible initial conditions for the equations, we prove a linear convergence rate. Moreover, we show linear convergence of the optimal states and optimal controls for the Cattaneo equation towards the ones for the heat equation. We present numerical results for %both, the forward and 
the optimal control problem confirming these linear convergence rates.
	\end{abstract}
	
\end{frontmatter}

\section{Introduction}
The Cattaneo model introduced in \cite{Cattaneo1958Sur} describes delayed heat transfer based on a modified Fourier's law with a delay time $\tau$. It is used to adjust the non-physical behaviour of infinite speed of propagation the heat equation exhibits. 

Accurate models for heat transfer are needed, e.g., for modelling medical treatments like thermoablation of tumors. Therefore, non-linear couplings of the heat equation with approximations of the radiative heat transfer equation are often employed. Various authors have investigated the Cattaneo equation and the effect of the delay parameter on the speed of propagation, also in the complex situation of thermoablation (see, e.g., \cite{Cattaneo1958Sur,LopezMolina2008Assessment,Tung2009Modeling,Li2017Non,Quintanilla2006note}). In real life applications, these models often depend on various unknown parameters. 
%Thus, there is the need to investigate the possibility of parameter identification, e.g., via a formulation as optimal control problem. 
Thus, a parameter identification via a formulation as inverse problem is of high interest. Recently, the Cattaneo model has been used in this context in \cite{Andres2022Cattaneo,Andres2021Improving}, where such parameter identification problems have been investigated. However, to the best of our knowledge, the asymptotic analysis of optimal control problems for the Cattaneo model is novel and has not yet been investigated in the previous literature.

The Cattaneo equation is a damped wave equation which reads
\begin{equation}
	\label{eq:cattaneo}
	\begin{alignedat}{2}
		\tau y_\tau'' + y_\tau' - \Delta y_\tau =\ &u_\tau \quad &&\text{ in } (0, T) \times \Omega, \\
		y_\tau =\ &0 \quad &&\text{ on } (0,T) \times \partial\Omega, \\
		y_\tau(0,\cdot) =\ &y_0 \quad &&\text{ in } \Omega, \\
		y_\tau'(0,\cdot) =\ &y_1 \quad &&\text{ in } \Omega,
	\end{alignedat}
\end{equation}
where $y_\tau$ denotes the temperature, the superscript $'$ denotes the partial time derivative, i.e., $y_\tau' := \frac{\partial}{\partial t} y_\tau$, $u_\tau$ is a heat source and $\tau>0$ is the delay parameter. Moreover, $\Omega\subset \mathbb{R}^d$ is a bounded Lipschitz domain with boundary $\partial\Omega$, $T>0$ is the time horizon of consideration, and $y_0$ and $y_1$ are the 
initial temperature distribution and initial temperature drift, respectively. 

Formally setting $\tau=0$ and neglecting the second initial condition yields the parabolic heat equation
\begin{equation}
	\label{eq:heat}
	\begin{alignedat}{2}
		y' - \Delta y =\ &u \quad &&\text{ in } (0, T) \times \Omega, \\
		y =\ &0 \quad &&\text{ on } (0,T) \times \partial\Omega,\\
		y(0,\cdot) =\ &y_0 \quad &&\text{ in } \Omega.
	\end{alignedat}
\end{equation}
Hence, the Cattaneo equation can be viewed as perturbation of the heat equation. Therefore, we expect that solutions of the Cattaneo equation \eqref{eq:cattaneo} converge to the ones of the heat equation \eqref{eq:heat} as $\tau \to 0$ and $u_\tau \to u$. In the remainder of this paper we write $y_\tau$ and $y$ for solutions corresponding to the Cattaneo and heat equation, respectively.

\begin{Remark}
	For the sake of simplicity, we restrict here to homogeneous Dirichlet boundary conditions. An extension of our results to other boundary conditions, such as Neumann or Robin boundary conditions, is straightforward.
\end{Remark}

\begin{Remark}
	Note that the Cattaneo equation can be considered a first order approximation for delayed heat transfer for the modified Fourier's law
	\begin{equation*}
		-\grad y(t,x) = q(t + \tau, x) = q(t,x) + \tau q'(t,x) + \mathcal{O}(\tau^2) \approx q(t,x) + \tau q'(t,x),
	\end{equation*}
	where $q$ denotes the heat flux. The heat equation is derived from the classical Fourier's law $q(t,x) = -\grad y(t,x)$. Therefore, there is a natural interest in studying the behavior of the Cattaneo equation for a vanishing delay time $\tau \to 0$. We refer the reader to \cite{Blauth2018Optimal} for a detailed derivation of the equations. 
\end{Remark}

The forward problems will serve as PDE constraints for the corresponding optimal control problems. For the Cattaneo equation this reads
\begin{equation}
	\label{eq:ocp_cattaneo}
	\begin{aligned}
		\min_{(y_\tau,u_\tau)}\ J_{\tau, \nu}(y_\tau,u_\tau) &= \frac{1}{2} \norm{y_\tau - y_d}{L^2(0, T; L^2(\Omega))}^2 + \frac{\nu \tau}{2} \norm{y_\tau(T) - y_d(T)}{L^2(\Omega)}^2 + \frac{\lambda}{2} \norm{u_\tau}{L^2(0,T;L^2(\Omega))}^2 \\
		&\text{ s.t. } \eqref{eq:cattaneo} \text{ and } u_\tau\in U_\mathrm{ad},
	\end{aligned}
\end{equation}
where $\lambda\geq 0$ is a parameter which allows to adjust the control costs, $y_d$ is the desired state, and $U_\mathrm{ad}$ is the set of admissible controls. Moreover, $\nu \in \Set{0,1}$ is a parameter that is used to distinguish whether a tracking of the state variable at the end of the time horizon is considered ($\nu = 1$) or not ($\nu=0$). %For this problem we again expect that for $\tau\to 0$ the solution of the optimization problem with the Cattaneo equation converges towards the solution of the corresponding problem with the heat equation as a constraint, i.e.,
We shall prove that for $\tau \to 0$ the optimal pairs $(y_\tau, u_\tau)$ converge linearly to the corresponding optimal state and optimal control of the optimal control problem constrained by the heat equation, i.e.,
\begin{equation}
	\label{eq:ocp_heat}
	\min_{(y,u)}\ J(y,u) = \frac{1}{2}\norm{y-y_d}{L^2(0,T;L^2(\Omega))}^2 + \frac{\lambda}{2} \norm{u}{L^2(0,T;L^2(\Omega))}^2 \quad \text{ s.t. } \eqref{eq:heat} \text{ and } u\in U_\mathrm{ad}.
\end{equation}
Naturally, we assume that the sets of admissible controls $U_\mathrm{ad}$ coincide for problems \eqref{eq:ocp_cattaneo} and \eqref{eq:ocp_heat} in order to investigate the asymptotic behavior of these problems in the limit $\tau \to 0$.

Note that the starting optimal control problem \eqref{eq:ocp_cattaneo} is constrained by a hyberbolic equation,  while the limit problem \eqref{eq:ocp_heat} is of parabolic nature. This change of character requires additional compatibility conditions and a thorough derivation of appropriate estimates to establish the linear convergence rate.

The rest of the paper is structured as follows: First, we introduce the notation and the Sobolev spaces used for the analysis in Section~\ref{sec:preliminaries}. Then, we investigate the well-posedness of both forward and optimal control problems for a fixed $\tau$ in Section~\ref{sec:well posedness}. The main result of this paper is given in Section~\ref{sec:asymptotic analysis}, where we prove that, under some additional assumptions, the solutions of \eqref{eq:cattaneo} and \eqref{eq:ocp_cattaneo} converge linearly to the ones of \eqref{eq:heat} and \eqref{eq:ocp_heat}, respectively. For the sake of better readability, some lengthy proofs of Section~\ref{sec:asymptotic analysis} are placed in Appendix~\ref{app:proofs}. Finally, in Section~\ref{sec:numerical results} we discuss numerical results for $\tau \to 0$ which underline the linear convergence proved analytically.  Conclusions and an outlook can be found in Section~\ref{sec:conclusion}.

\subsection{Preliminaries \& Notations}
\label{sec:preliminaries}

Our notation is based on \cite{Hinze2009Optimization}: $X$ is a Banach space and $H$ a Hilbert space. The norm on $X$ is denoted by $\norm{\cdot}{X}$ and the scalar product on $H$ by $\inner{\cdot}{\cdot}{H}$. In particular, we define $\scalar{\cdot}{\cdot} := \inner{\cdot}{\cdot}{L^2(\Omega)}$. The dual space of $X$ is denoted by $X^*$ and the duality pairing by $\varphi (y) = \dual{\varphi}{y}{X^*,X}$ for $\varphi \in X^*$ and $y\in X$. For the dual space of $H^1_0(\Omega)$ we write $H^{-1}(\Omega)$ and for better readability we define $\dual{\varphi}{y}{} := \dual{\varphi}{y}{H^{-1}(\Omega), H^1_0(\Omega)}$. We write $y_k \rightharpoonup y$ in $X$ for the weak convergence of the sequence $(y_k) \subset X$ to $y \in X$. With $C$ we denote a generic and positive constant, possibly having different values and we denote by $C=C(\alpha)$ the dependence of the constant $C$ on a parameter $\alpha$. The space $C^\infty_0(\Omega)$ denotes the space of infinitely differentiable functions having compact support in $\Omega$. For our analysis we also need Lebesgue spaces of vector valued functions $L^p(0,T;X)$ that are defined, e.g., in \cite{Hinze2009Optimization}. 
%call a function $w\in L^1(0,T;X)$ the $k$-th order weak time derivative of $y$ if it holds
%$$ \int\limits_{0}^{T} y(t) \frac{d^k}{dt^k} \varphi(t) \dx{t} = \left(-1 \right)^k \int\limits_{0}^{T} w(t) \varphi(t) \dx{t} \quad \text{ for all } \varphi \in C^\infty_0((0,T)),
%$$
%and in this case we 
We denote by $\frac{d^k}{dt^k} y$ the $k$-th order weak time derivative of $y$. In particular, we write $\frac{d}{dt} y := y'$ as well as $\frac{d^2}{dt^2} y := y''$ for the weak time derivatives. For $m\in \mathbb{N}$ and $1\leq p \leq \infty$ we define the space 
$$W^{m,p}(0,T;X):= \Set{y\in L^p(0,T;X) | \frac{d^k}{dt^k} y \in L^p(0,T;X) \text{ for } k=1,\dots,m},
$$ 
i.e., the space of all $y\in L^p(0,T;X)$ whose weak time derivatives are in $L^p(0,T;X)$ up to order $m$, equipped with the usual Sobolev norm. For $p=2$ we also write $H^m(0,T;H) := W^{m,2}(0,T;H)$ and recall that the spaces $H^m(0,T;H)$ are Hilbert spaces. Throughout the rest of this paper we omit the explicit time dependence of the functions in several derivations for a better readability.
Furthermore, the space $C^k([0,T];X)$ is defined as the space of $k$-times continuously differentiable functions on $[0,T]$ with values in $X$ and the space $C^\infty_0((0,T);X)$ is defined as the space of infinitely differentiable functions with compact support in $(0,T)$ and values in $X$.

\section{Well-Posedness of the State Eqations and Optimal Control}
\label{sec:well posedness}

In this section, we briefly discuss the well-posedness of the optimal control problems \eqref{eq:ocp_cattaneo} and \eqref{eq:ocp_heat} with $\tau >0$ fixed and refer the reader to \cite{Blauth2018Optimal} for a more detailed discussion. In order to simplify the notation, throughout the rest of the paper we denote with $a$ the usual bilinear form for the %Laplacian 
Laplace operator with homogeneous Dirichlet data, i.e.,
\begin{equation}
	\label{eq:definition_bilinear_form}
	a\colon H^1_0(\Omega)\times H^1_0(\Omega) \to \mathbb{R}; \quad (y,v) \mapsto a[y,v] := (\grad y,\grad v).
\end{equation}
It is well-known that the bilinear form $a$ is symmetric, continuous, and coercive (see, e.g., \cite{Evans2010Partial}).
%weiter nach vorna: Finally, we also want to remark that we restrict ourselves to homogeneous Dirichlet boundary conditions throughout the rest of this paper, for the sake of simplicity. An extension of our results to other \qe{standard} boundary conditions, e.g., Neumann or Robin boundary conditions, is straightforward.

\subsection{Well-Posedness of the Cattaneo Equation}
\label{ssec:well-posed}

%In order to investigate optimal control problems with the Cattaneo equation as constraint, we introduce a weak formulation of the equation and investigate its well-posedness. 
We seek a weak solution of the Cattaneo equation \eqref{eq:cattaneo} in the space
\begin{equation*}
	Y(0,T) := \Set{v\in L^2(0,T;H^1_0(\Omega)) | v'\in L^2(0,T;L^2(\Omega)) \text{ and } v''\in L^2(0,T;H^{-1}(\Omega))}.
\end{equation*}
Particularly, a weak solution of the Cattaneo equation satisfies the variational problem
\begin{equation}
	\label{eq:weak_cattaneo}
	\text{Find }y_\tau \in Y(0,T) \text{ s.t. } \quad \int\limits_{0}^{T} \tau \dual{y_\tau''}{v}{} + \scalar{y_\tau'}{v} + a[y_\tau,v] \dx{t} = \int\limits_{0}^{T} \scalar{u_\tau}{v} \dx{t}
	\quad\text{ for all } v\in L^2(0,T;H^1_0(\Omega))
\end{equation}
and the initial conditions $y_\tau(0)=y_0$ and $y_\tau'(0) = y_1$. 

\begin{Remark}
Note that if we have $y_\tau\in Y(0,T)$, we also have $y_\tau\in C([0,T];L^2(\Omega))$ and $y_\tau'\in C([0,T];H^{-1}(\Omega))$ (cf. \cite{Blauth2018Optimal} or \cite{Evans2010Partial}). Therefore, the time evaluation of $y_\tau$ and $y_\tau'$ is well-defined and justifies the initial conditions. 
\end{Remark}

The well-posedness of this weak formulation can be proved with the help of the Faedo-Galerkin method (see, e.g., \cite{Evans2010Partial,Lions1971Optimal,Faedo1949Un}). We obtain the following result, which is proved in \cite[Theorem~3.12]{Blauth2018Optimal}. 
\begin{Theorem}
	\label{thm:well_posedness_cattaneo}
	Let $\tau>0$, $y_0\in H^1_0(\Omega)$, $y_1\in L^2(\Omega)$ and $u_\tau\in L^2(0,T;L^2(\Omega))$. Then, there exists a unique weak solution $y_\tau\in Y(0,T)$ of the Cattaneo equation \eqref{eq:cattaneo} which depends continuously on the data. In fact, there exists a constant $C=C(\tau)>0$ such that
	\begin{equation}
		\label{eq:energy_estimate_1}
		\begin{aligned}
			\norm{y_\tau}{Y(0,T)}^2 =\ &\norm{y_\tau}{L^2(0,T;H^1_0(\Omega))}^2 + \norm{y_\tau'}{L^2(0,T;L^2(\Omega))}^2 + \norm{y_\tau''}{L^2(0,T;H^{-1}(\Omega))}^2 \\
			\leq\ &C(\tau) \left( \norm{u_\tau}{L^2(0,T;L^2(\Omega))}^2 + \norm{y_0}{H^1_0(\Omega)}^2 + \norm{y_1}{L^2(\Omega)}^2 \right).
		\end{aligned}
	\end{equation}		
\end{Theorem}

\begin{Remark}\label{rem:estimate_not_sufficient}
	In \cite{Blauth2018Optimal} it is shown that $C(\tau) \to \infty$ as $\tau \to 0$ for the constant $C(\tau)$ in \eqref{eq:energy_estimate_1}. Hence, in order to investigate the limit $\tau\to 0$ we need stronger energy estimates, independent of $\tau$ (see Section~\ref{sec:asymptotic analysis}).
\end{Remark}

\subsection{Well-Posedness of the Heat Equation}

As we are interested in the limit $\tau \to 0$, we introduce the weak formulation of the heat equation (cf.\ \cite{Evans2010Partial}) for completeness. We seek a weak solution of the heat equation in the space
\begin{equation*}
	W(0,T) := \set{v \in L^2(0,T;H^1_0(\Omega)) | v'\in L^2(0,T;L^2(\Omega))}.
\end{equation*}
Particularly, a weak solution of the heat equation satisfies the variational problem
\begin{equation}
	\label{eq:weak_heat}
	\text{Find }y \in W(0,T) \text{ s.t. } \quad \int\limits_{0}^{T} \scalar{y'}{v} + a[y, v] \dx{t} = \int\limits_{0}^{T} \scalar{u}{v} \dx{t} \quad \text{ for all } v\in L^2(0,T;H^1_0(\Omega))
\end{equation}
and the initial condition $y(0) = y_0$, where the bilinear form $a$ is again given by \eqref{eq:definition_bilinear_form}.
The following well-posedness result is well-known, see e.g.~\cite{Evans2010Partial,Lions1971Optimal}.
\begin{Theorem}
	\label{thm:well_posedness_heat}
	%The heat equation \eqref{eq:heat} is well-posed in the following sense: 
	Let $u\in L^2(0,T;L^2(\Omega))$ and $y_0 \in H^1_0(\Omega)$. Then, there exists a unique weak solution of the heat equation \eqref{eq:heat} that depends continuously on the data. Particularly, there exists a constant $C>0$ so that
	\begin{equation*}
		\norm{y}{W(0,T)}^2 = \norm{y}{L^2(0,T;H^1_0(\Omega))}^2 + \norm{y'}{L^2(0,T;L^2(\Omega))}^2 \leq C\left( \norm{u}{L^2(0,T;L^2(\Omega))}^2 + \norm{y_0}{H^1_0(\Omega)}^2 \right).
	\end{equation*}
\end{Theorem}

\begin{Remark}
	Due to the higher regularity required for the well-posedness of the Cattaneo equation and our interest in the limit $\tau \to 0$, we use a slightly different state space $W(0,T)$ compared to the literature (see, e.g., \cite{Evans2010Partial,Lions1971Optimal}), where usually the space
	$\set{v\in L^2(0,T;H^1_0(\Omega)) | v' \in L^2(0,T;H^{-1}(\Omega))}$
	is used as solution space for the heat equation.
\end{Remark}

\subsection{Optimal Control Problems}

As already mentioned in the introduction, we are interested in the asymptotic behaviour for $\tau\to 0$ in the optimal control context. Therefore, we state the optimal control problems subject to the Cattaneo and the heat equation in the following. Throughout this section, we need the following assumption:
\begin{Assumption}
	\label{ass:regularity}
	Let $y_0 \in H^3(\Omega) \cap H^1_0(\Omega)$, $0 < M < \infty$, and $\bar{u} \in H^1(\Omega)$ be fixed. The set of admissible controls $U_{\mathrm{ad}}$ is a weakly closed and convex subset of the set
	\begin{equation*}
		\hat{U} := \Set{u \in H^1(0, T; H^1(\Omega)) | \norm{u}{H^1(0,T;H^1(\Omega))} \leq M \text{ and } u(0) = \bar{u}}.	
	\end{equation*}
	Further, we assume that
	\begin{equation}
		\label{eq:initial_condition}
		\laplace y_0 + \bar{u} \in H^1_0(\Omega),
	\end{equation}
	and that $y_d \in H^1(0,T;H^1_0(\Omega))$.
\end{Assumption}

\begin{Remark}
	Note that for $u\in \hat{U} \subset H^1(0,T;H^1(\Omega))$ we are allowed to make point evaluations in time due to the embedding $H^1(0,T;H^1(\Omega)) \embedding C([0,T];H^1(\Omega))$. Therefore, we have $u(t) \in H^1(\Omega)$ for all $t\in [0,T]$ and, hence, the evaluation of initial and terminal values of $u$ and the definition of $\hat{U}$ is justified. Analogously, we get that $y_d \in H^1(0,T;H^1_0(\Omega)) \embedding C([0,T];H^1_0(\Omega))$ so that $y_d(t) \in H^1_0(\Omega)$ for all $t\in [0,T]$.
\end{Remark}

\subsubsection{Optimal Control of the Cattaneo equation}
For $\tau > 0$ we consider the family of optimal control problems given by \eqref{eq:ocp_cattaneo} where we consider $y_d \in H^1(0,T;H^1_0(\Omega))$, $y_0 \in H^3(\Omega) \cap H^1_0(\Omega)$, and $y_1 = \laplace y_0 + \bar{u} \in H^1_0(\Omega)$. The motivation for this choice of $y_1$ for the Cattaneo equation is discussed later in Section~\ref{ssec:intro_asymp}. As state space for the optimal control problem we use the space $Y(0,T)$.  Thanks to the continuous embedding $Y(0,T) \embedding L^2(0,T;L^2(\Omega))$ the cost functional in \eqref{eq:ocp_cattaneo} is well-defined. Using standard techniques for linear-quadratic optimal control problems we obtain the following existence and uniqueness result by an application of \cite[Theorem~1.43]{Hinze2009Optimization} (cf.~\cite[Theorem~4.2]{Blauth2018Optimal}).
%\begin{equation}
%	\label{eq:distributed_control_cattaneo}
%	\min_{(y_\tau,u_\tau)}\ J(y_\tau,u_\tau) = \frac{1}{2} \norm{y_\tau - y_d}{L^2(0,T;L^2(\Omega))}^2 + \frac{\lambda}{2} \norm{u_\tau}{L^2(0,T;L^2(\Omega))}^2 \quad \text{ s.t. \eqref{eq:cattaneo} and }u_\tau\in U_\mathrm{ad}, 
%\end{equation}
%where $\lambda \geq 0$ is a  parameter which allows to adjust the control costs, $y_d\in L^2(0,T;L^2(\Omega))$ is the desired state, and $y_0 \in H^1_0(\Omega)$ and $y_1\in L^2(\Omega)$ are given initial conditions for \eqref{eq:cattaneo}. 

\begin{Theorem}
	\label{thm:optimality_cattaneo}
	Let $\tau>0$, $\nu \in \set{0,1}$, $y_0 \in H^3(\Omega) \cap H^1_0(\Omega)$, $y_1 = \laplace y_0 + \bar{u} \in H^1_0(\Omega)$, $y_d\in H^1(0,T;H^1_0(\Omega))$, $\lambda \geq 0$, and let Assumption~\ref{ass:regularity} hold. Then, problem \eqref{eq:ocp_cattaneo} has a unique optimal control $u_\tau^*\in U_\mathrm{ad}$.
\end{Theorem}

One can easily derive the first order optimality conditions for problem \eqref{eq:ocp_cattaneo} using the well-known adjoint approach (see, e.g., \cite{Hinze2009Optimization} or \cite{Troeltzsch2010Optimal}), which is done in \cite[Theorem~4.6]{Blauth2018Optimal}, and we summarize them in the following.
\begin{Theorem}
	\label{thm:fonc_cattaneo}
	Let $\tau>0$, $\nu \in \set{0,1}$, $y_0 \in H^3(\Omega) \cap H^1_0(\Omega)$, $y_1 = \laplace y_0 + \bar{u} \in H^1_0(\Omega)$, $y_d\in H^1(0,T;H^1_0(\Omega))$, $\lambda\geq 0$, and let Assumption~\ref{ass:regularity} hold. The first order optimality conditions for $u_\tau^*\in U_\mathrm{ad}$ being an optimal control of problem \eqref{eq:ocp_cattaneo}	with corresponding optimal state $y_\tau^*\in Y(0,T)$ and optimal adjoint state $p_\tau^*\in Y(0,T)$ are given by:

	\begin{minipage}{0.45\textwidth}
		\begin{equation*}
			\begin{alignedat}{2}
				\tau (y_\tau^*)'' + (y_\tau^*)' - \laplace y_\tau^* =\ &u_\tau^* \quad &&\text{ in } (0,T) \times \Omega, \\
				y_\tau^* =\ &0 \quad &&\text{ on } (0,T) \times \partial\Omega,\\
				y_\tau^*(0,\cdot) =\ &y_0 \quad &&\text{ in } \Omega,\\
				(y_\tau^*)'(0,\cdot) =\ &y_1 \quad &&\text{ in } \Omega,
			\end{alignedat}
		\end{equation*}
	\end{minipage}%
	\begin{minipage}{0.55\textwidth}
		\begin{equation*}
			\begin{alignedat}{2}
				\tau (p_\tau^*)'' - (p_\tau^*)' - \laplace p_\tau^* =\ &y_\tau^* - y_d \quad &&\text{ in } (0,T) \times \Omega,\\
				p_\tau^* =\ &0 \quad &&\text{ on } (0,T) \times \partial \Omega,\\
				p_\tau^*(T,\cdot) =\ &0 \quad &&\text{ in } \Omega,\\
				-(p_\tau^*)'(T,\cdot) =\ &\nu \left( y_\tau^*(T) - y_d(T)  \right) \quad &&\text{ in } \Omega,
			\end{alignedat}
		\end{equation*}
	\end{minipage}
	\begin{equation*}
		\inner{p_\tau^* + \lambda u_\tau^*}{u-u_\tau^*}{L^2(0,T;L^2(\Omega))} \geq 0 \quad \text{for all } u \in U_\mathrm{ad}. 
		%\quad \text{ or }\quad p_\tau^* + \lambda u_\tau^* = 0 \quad \text{ in case } U_\mathrm{ad} = U. 
	\end{equation*}
\end{Theorem}
These conditions are already sufficient since our cost functional is convex. 

\begin{Remark}
	As the Cattaneo equation is a damped wave equation for a fixed delay parameter $\tau$, the investigations regarding well-posedness and optimal control (cf.~Theorems~\ref{thm:well_posedness_cattaneo}, \ref{thm:optimality_cattaneo}, and \ref{thm:fonc_cattaneo}) can be carried out very similarly to the analysis performed in \cite[Chapters~3 and~4]{Lions1971Optimal}.
\end{Remark}

\subsubsection{Optimal Control of the Heat Equation}
For later use we summarize the results concerning the optimal control of the heat equation. A detailed discussion can be found in, e.g., \cite{Hinze2009Optimization,Lions1971Optimal,Troeltzsch2010Optimal}. We consider problem \eqref{eq:ocp_heat} and use the same set of admissible controls as for the optimal control of the Cattaneo equation. The state space is given by $W(0,T)$ and due to the embedding $W(0,T) \embedding L^2(0,T;L^2(\Omega))$, the cost functional is well-defined. 
\begin{Theorem}
	\label{thm:optimality_heat}
	Let $\lambda\geq0$, $y_0\in H^3(\Omega) \cap H^1_0(\Omega)$, $y_d\in H^1(0,T;H^1_0(\Omega))$, and let Assumption~\ref{ass:regularity} hold. Then, problem \eqref{eq:ocp_heat} has a unique optimal control $u^*\in U_\mathrm{ad}$.
\end{Theorem}
\begin{Theorem}
	\label{thm:fonc_heat}
	Let $\lambda\geq0$, $y_0 \in H^3(\Omega) \cap H^1_0(\Omega)$, $y_d\in H^1(0,T;H^1_0(\Omega))$, and let Assumption~\ref{ass:regularity} hold. The first order optimality conditions for $u^*\in U_\mathrm{ad}$ being an optimal control of problem \eqref{eq:ocp_heat} with corresponding optimal state $y^*\in W(0,T)$ and optimal adjoint state $p^*\in W(0,T)$ are given by: \\
	
	\begin{minipage}{0.5\textwidth}
		\begin{equation*}
			\begin{alignedat}{2}
				(y^*)' - \laplace y^* &= u^* \quad &&\text{ in } (0,T) \times \Omega,\\
				y^* &= 0 \quad &&\text{ on } (0,T) \times \partial\Omega,\\
				y^*(0,\cdot) &= y_0 \quad &&\text{ in } \Omega,
			\end{alignedat}
		\end{equation*}
	\end{minipage}%
	\begin{minipage}{0.5\textwidth}
		\begin{equation*}
			\begin{alignedat}{2}
			-(p^*)' - \laplace p^* &= y^* - y_d \quad &&\text{ in } (0,T) \times \Omega,\\
			p^* &= 0 \quad &&\text{ on } (0,T) \times \partial\Omega,\\
			p^*(T,\cdot) &= 0 \quad &&\text{ in } \Omega,
			\end{alignedat}
		\end{equation*}
	\end{minipage}
	\begin{equation*}
		\inner{p^* + \lambda u^*}{u - u^*}{L^2(0,T;L^2(\Omega))} \geq 0 \quad \text{ for all } u \in U_\mathrm{ad}. %\quad \text{ or } \quad p^* + \lambda u^* = 0 \quad \text{ in case }U_\mathrm{ad} = U.
	\end{equation*}
\end{Theorem}

\begin{Remark}
	\label{rem:regularity}
	The results regarding the optimal control of the Cattaneo and heat equations are also true under less restrictive regularity assumptions on the data in the case $\nu = 0$. In this case, the same results hold for $y_0 \in H^1_0(\Omega)$, $y_1 \in L^2(\Omega)$, $y_d \in L^2(0,T;L^2(\Omega))$ and $U_\mathrm{ad}$ being a closed and convex subset of $L^2(0,T;L^2(\Omega))$ which is bounded in case $\lambda = 0$ (cf. \cite{Blauth2018Optimal}).
\end{Remark}

\section{Asymptotic Analysis}
\label{sec:asymptotic analysis}

As the estimate of Theorem~\ref{thm:well_posedness_cattaneo} is not sufficient to pass to the limit $\tau \to 0$ (cf.\ Remark~\ref{rem:estimate_not_sufficient}), we state stronger energy estimates in the following.
% testing the approximate solution with both its time derivative and itself, and using the Cauchy-Schwarz and Young's inequality in order to estimate the terms. 
\begin{Lemma}
	\label{lem:boundedness}
	For every $\tau>0$ the weak solution $y_\tau\in Y(0,T)$ of the Cattaneo equation \eqref{eq:cattaneo} from Theorem~\eqref{thm:well_posedness_cattaneo} satisfies
	%Additionally, there exists a constant $C>0$ independent of $\tau$ such that
	\begin{equation*}
		\norm{y_\tau}{L^2(0,T;H^1_0(\Omega))}^2 + \norm{y_\tau'}{L^2(0,T;L^2(\Omega))}^2 \leq C \left( \norm{u_\tau}{L^2(0,T;L^2(\Omega))}^2 + \norm{y_0}{H^1_0(\Omega)}^2 + \norm{y_1}{L^2(\Omega)}^2 \right)
	\end{equation*}
	for a constant $C>0$ independent of $\tau.$
\end{Lemma}
\begin{proof}
	The proof can be found in Appendix~\ref{proof:boundedness}.
\end{proof}

This result is already sufficient to prove the convergence of solutions of the Cattaneo equation to the ones of the heat equation in the limit $\tau \to 0$ as well as the convergence of the optimal controls and states, albeit without a convergence rate, but with minimal regularity assumptions. We investigate this in the following, where we only consider the case $\nu = 0$ since this allows us to use the minimal regularity assumptions (cf.~Theorem~\ref{thm:convergence_ocp_rateless} and Remark~\ref{rem:regularity}).

\begin{Theorem}
	\label{thm:convergence_rateless}
	Let $\{\tau_i\}\subset \mathbb{R^+}$ be a sequence with $\lim_{i\to \infty} \tau_i = 0$. Let $u\in L^2(0,T;L^2(\Omega))$, $y_0\in H^1_0(\Omega)$, $y_1\in L^2(\Omega)$ be given.
	Let $\{y_{\tau_i}\} \subset Y(0,T)$ be the sequence of unique weak solutions of the Cattaneo equation~\eqref{eq:cattaneo} corresponding to $\{\tau_i\}$
	and let $y\in W(0,T)$ be the unique weak solution of the heat equation~\eqref{eq:heat}.
	Then, it holds
	\begin{equation*}
		y_{\tau_i} \rightarrow y   \text{ in }L^2(0,T;L^2(\Omega)), \qquad
		y_{\tau_i} \rightharpoonup y  \text{ in }L^2(0,T;H^1_0(\Omega))\quad\text{ and }\quad
		y_{\tau_i}'\rightharpoonup y'  \text{ in }L^2(0,T;L^2(\Omega)).
	\end{equation*}
\end{Theorem}
\begin{proof}
	The proof can be found in Appendix~\ref{proof:convergence_rateless}.
\end{proof}

\begin{Theorem}
	\label{thm:convergence_ocp_rateless}
	Let $\nu=0$, $\lambda> 0$, $y_0\in H^1_0(\Omega)$, $y_1\in L^2(\Omega)$, $y_d\in L^2(0,T;L^2(\Omega))$, and let Assumption~\ref{ass:regularity} hold. Let $\{\tau_i\}\subset\mathbb{R^+}$ be a sequence with  $\lim_{i\to\infty}\tau_i = 0$. 
	
	For $i\in \mathbb{N}$ we denote by $u_{\tau_i}^*\in U_\mathrm{ad}$ the unique minimizer of problem \eqref{eq:ocp_cattaneo} and by $y_{\tau_i}^*$  its corresponding optimal state with weak time derivative $(y_{\tau_i}^*)'$. Further, let $u^*$ be the unique minimizer of problem \eqref{eq:ocp_heat} with corresponding optimal state $y^*$ and weak time derivative $(y^*)'$.
	
	Then, for the optimal states, it holds
	\begin{equation*}
		y_{\tau_i}^* \rightarrow y^* \text{ in } L^2(0,T;L^2(\Omega)),\qquad  y_{\tau_i}^* \rightharpoonup y^* \text{ in } L^2(0,T;H^1_0(\Omega)), \qquad (y_{\tau_i}^*)' \rightharpoonup (y^*)' \text{ in } L^2(0,T;L^2(\Omega))
	\end{equation*}
	and for the optimal controls, it holds
	\begin{equation*}
		u_{\tau_i}^*\rightarrow u^* \text{ in } L^2(0,T;L^2(\Omega)).
	\end{equation*}
\end{Theorem}
\begin{proof}
	The proof can be found in Appendix~\ref{proof:convergence_ocp_rateless}.
\end{proof}

\subsection{Compatibility Condition}
\label{ssec:intro_asymp}

The previous results show that we indeed have convergence of solutions of the Cattaneo equation to the ones of the heat equation for both the forward and the optimal control problem. In the following we additionally provide a convergence rate. For the stronger result, we require higher regularity assumptions as well as a compatibility condition, which we introduce in the following.

Under Assumption~\ref{ass:regularity}, we can differentiate the heat equation with respect to time to obtain
\begin{equation}
	\label{eq:derivative_heat}
	\begin{alignedat}{2}
		y'' - \laplace y' &= u' \quad &&\text{ in } (0,T) \times \Omega,\\
		y' &= 0 \quad &&\text{ on } (0,T) \times \partial\Omega,\\
		y'(0,\cdot) &= \laplace y_0 + \bar u \quad &&\text{ in } \Omega.
	\end{alignedat}
\end{equation}
Note that $y_0\in H^3(\Omega) \cap H^1_0(\Omega)$ yields $\laplace y_0 \in H^1(\Omega)$, and thanks to Assumption~\ref{ass:regularity} the initial condition of equation \eqref{eq:derivative_heat} is in $H^1_0(\Omega)$. Furthermore, since $u'\in L^2(0,T;H^1(\Omega)) \embedding L^2(0,T;L^2(\Omega))$ we can apply Theorem~\ref{thm:well_posedness_heat} and obtain a unique weak solution $y' \in W(0,T)$ of \eqref{eq:derivative_heat}. It is easy to see that $y'$ is indeed the weak time derivative of $y$ (cf.\ \cite{Evans2010Partial}). Finally, we observe that $y\in C([0,T];H^1_0(\Omega))$ and $y'\in C([0,T];L^2(\Omega))$. Thus, we can evaluate the heat equation for all $t\in [0,T]$ and, in particular, for $t=0$ we obtain
\begin{equation*}
	y'(0) - \laplace y(0) = y'(0) - \laplace y_0 = \bar u,
\end{equation*}
which coincides with the initial condition in \eqref{eq:derivative_heat}. This justifies to choose the second initial condition for the Cattaneo equation as 
\begin{equation}
	\label{eq:compatibility_condition}
	y_1 := \laplace y_0 + \bar u
\end{equation}
in order to be compatible with the heat equation.

\begin{Remark}
	\label{rem:compatibility_adjoint}
	The proof of the convergence rate for the optimal states and controls makes use of the convergence of the adjoint states $p_\tau$. Clearly, we cannot assume compatibility conditions for the adjoints, since the equations are determined by the cost functional and the state problem at hand. However, using $\nu = 1$ in the cost functional of \eqref{eq:ocp_cattaneo} naturally yields an analogous compatibility condition also for the adjoint Cattaneo equation, which is discussed later.
\end{Remark}

\subsection{Linear Convergence of the states}
\label{ssec:forward_asymptotic}
We begin with the proof of linear convergence of the states for given right-hand sides $u_\tau$ and $u.$
\begin{Theorem}
	\label{thm:linear_convergence_state}
	Let $y_0 \in H^3(\Omega) \cap H^1_0(\Omega)$, $u, u_\tau \in \hat{U}$ such that $y_1 = \laplace y_0 + \bar{u} \in H^1_0(\Omega)$. Let $y_\tau$ be the weak solution of \eqref{eq:cattaneo} and $y$ the weak solution of \eqref{eq:heat} corresponding to $u_\tau$ and $u$, respectively. Then, there exists a constant $C>0$, independent of $\tau$, such that
	\begin{equation*}
		\norm{y_\tau - y}{L^2(0,T;H^1_0(\Omega))} \leq C \norm{u_\tau - u}{L^2(0,T;L^2(\Omega))} + \tau C \left( \norm{u_\tau}{H^1(0,T;H^1(\Omega))} + \norm{y_0}{H^3(\Omega)} \right).
	\end{equation*}
\end{Theorem}
\begin{proof}
	Thanks to the assumptions made in Section~\ref{ssec:intro_asymp}, we are allowed to differentiate the Cattaneo equation with respect to time and obtain
	\begin{equation}
		\label{eq:derivative_cattaneo}
		\begin{alignedat}{2}
			\tau y_\tau''' + y_\tau'' -\laplace y_\tau' &= u_\tau' \quad &&\text{ in } (0,T) \times \Omega,\\
			y_\tau' &= 0 \quad &&\text{ on } (0,T) \times \partial\Omega,\\
			y_\tau'(0,\cdot) &= y_1 \quad &&\text{ in } \Omega,\\
			y_\tau''(0,\cdot) &= \frac{1}{\tau} \left( \laplace y_0 + \bar u - y_1 \right) = 0 \quad &&\text{ in } \Omega,
		\end{alignedat}
	\end{equation}
	where we used the compatibility condition~\eqref{eq:compatibility_condition}. %$y_1 = \laplace y_0 + u_\tau(0)$. 
	We can apply Theorem~\ref{thm:well_posedness_cattaneo} since $y_1 \in H^1_0(\Omega)$ and $u_\tau' \in L^2(0,T;H^1(\Omega)) \embedding L^2(0,T;L^2(\Omega))$. This yields the existence and uniqueness of a weak solution $y_\tau' \in Y(0,T)$ of \eqref{eq:derivative_cattaneo}, which is also the weak time derivative of $y_\tau$. Moreover, Lemma~\ref{lem:boundedness} yields the existence of a constant $C>0$, independent of $\tau$, such that
	% HIER
	\begin{equation}
		\label{eq:energy_derivative}
		\norm{y_\tau'}{L^2(0,T;H^1_0(\Omega))}^2 + \norm{y_\tau''}{L^2(0,T;L^2(\Omega))}^2 %\leq\ &C\left( \norm{u_\tau'}{L^2(0,T;H^1(\Omega))}^2 + \norm{y_1}{H^1_0(\Omega)}^2 \right)\\
		\leq C\left( \norm{u_\tau}{H^1(0,T;H^1(\Omega))}^2 + \norm{y_0}{H^3(\Omega)}^2 \right),
	\end{equation}
	where we used the boundedness of the operator $\Delta \colon H^3(\Omega) \to H^1(\Omega)$ and the embedding $H^1(0,T;H^1(\Omega))\embedding C([0,T];H^1(\Omega))$. Additionally, $u_\tau$ is bounded in $H^1(0,T;H^1(\Omega))$ independently of $\tau$ due to the construction of $\hat{U}$. Therefore, combining Lemma~\ref{lem:boundedness} and \eqref{eq:energy_derivative} yields the following energy estimate
	\begin{equation}
		\label{eq:energy_estimate_complete}
		\norm{y_\tau}{L^2(0,T;H^1_0(\Omega))}^2 + \norm{y_\tau'}{L^2(0,T;H^1_0(\Omega))}^2 + \norm{y_\tau''}{L^2(0,T;L^2(\Omega))}^2 \leq C\left( \norm{u_\tau}{H^1(0,T;H^1(\Omega))}^2 + \norm{y_0}{H^3(\Omega)}^2 \right),
	\end{equation} 
	where $C>0$ is independent of $\tau$. We take the difference of the weak formulation of the Cattaneo equation \eqref{eq:weak_cattaneo} and the weak formulation of the heat equation \eqref{eq:weak_heat} to observe
	\begin{equation*}
		\int\limits_0^T \tau (y_\tau'',v) + (y_\tau' - y',v) + a[y_\tau - y, v] \dx{t} = \int\limits_0^T (u_\tau - u,v) \dx{t} \quad \text{for all }v\in L^2(0,T;H^1_0(\Omega)).
	\end{equation*}
	%Note that we have omitted the time dependency of the functions for better readability. 
	Choosing $v = y_\tau - y$ reveals
	\begin{equation*}
		\int\limits_0^T \tau (y_\tau'',y_\tau - y) + \frac{1}{2} \frac{d}{dt} \norm{y_\tau - y}{L^2(\Omega)}^2 + a[y_\tau - y, y_\tau - y] \dx{t} = \int\limits_0^T (u_\tau - u,y_\tau -  y) \dx{t}.
	\end{equation*}
	Using the coercivity of $a$ together with the Cauchy-Schwarz inequality and some rearrangements gives
	\begin{equation}
		\label{eq:norm_difference}
		\norm{y_\tau - y}{L^2(0,T;H^1_0(\Omega))}^2 %\leq\ &C\left( \norm{u_\tau - u}{L^2(0,T;L^2(\Omega))} + \tau \norm{y_\tau''}{L^2(0,T;L^2(\Omega))} \right)  \norm{y_\tau - y}{L^2(0,T;L^2(\Omega))} \\
		\leq C\left( \norm{u_\tau - u}{L^2(0,T;L^2(\Omega))} + \tau \norm{y_\tau''}{L^2(0,T;L^2(\Omega))} \right) \norm{y_\tau - y}{L^2(0,T;H^1_0(\Omega))},
	\end{equation}
	where we have used that $\norm{y_\tau(T) - y(T)}{L^2(\Omega)} \geq 0$, $y_\tau(0) = y_0 = y(0)$ and the embedding $L^2(0,T;H^1_0(\Omega)) \embedding L^2(0,T;L^2(\Omega))$. Dividing \eqref{eq:norm_difference} by $\norm{y_\tau - y}{L^2(0,T;H^1_0(\Omega))}$ and using the energy estimate \eqref{eq:energy_estimate_complete} results in
	\begin{equation*}
		\norm{y_\tau - y}{L^2(0,T;H^1_0(\Omega))} \leq C\left( \norm{u_\tau - u}{L^2(0,T;L^2(\Omega))} + \tau \left( \norm{u_\tau}{H^1(0,T;H^1(\Omega))} + \norm{y_0}{H^3(\Omega)} \right) \right),
	\end{equation*}
which completes the proof.
\end{proof}

The convergence of the states is an immediate consequence.
\begin{Corollary}
	\label{cor:linear_convergence_state}
	Let $y_0\in H^3(\Omega) \cap H^1_0(\Omega)$, $u_\tau, u \in \hat{U}$ such that $y_1 = \laplace y_0 +\bar u \in H^1_0(\Omega)$. Furthermore, let $u_\tau \to u$ in $H^1(0,T;H^1(\Omega))$ as $\tau \to 0$. Then, we have
	$$y_\tau \to y \;\text{ in } L^2(0,T;H^1_0(\Omega)) \; \text{ as }\tau \to 0.$$ 
	Moreover, if there exists some constant $\tilde{C} > 0$ independent of $\tau$ such that $\norm{u_\tau - u}{L^2(0,T;L^2(\Omega))} \leq \tau \tilde{C}$ for all sufficiently small $\tau$, then there exists $C>0$ independent of $\tau$ such that $\norm{y_\tau - y}{L^2(0,T;H^1_0(\Omega))} \leq C\tau$, i.e., we have linear convergence of the states.
\end{Corollary}
\begin{proof}
	The proof follows directly from Theorem~\ref{thm:linear_convergence_state}.
\end{proof}

\subsection{Linear Convergence of the Optimal States and Optimal Controls}

In the following we only consider the case $\nu=1$ as this ensures the compatibility condition for the adjoint equation, as is discussed in the following. Under Assumption~\ref{ass:regularity}, we can differentiate the adjoint heat equation with respect to time to obtain
\begin{equation}
	\label{eq:adjoint_heat_derivative}
	\begin{alignedat}{2}
		-p'' - \laplace p' &= y' - y_d' \quad &&\text{ in } (0,T)\times\Omega,\\
		p' &= 0 \quad &&\text{ on } (0,T)\times\partial\Omega,\\
		-p'(T,\cdot) &= y(T,\cdot) - y_d(T,\cdot) \quad &&\text{ in }\Omega.
	\end{alignedat}
\end{equation}
By linearity, the adjoint equation has the same structure as the state equation. Moreover, we observe that $y'-y_d' \in L^2(0,T;H^1_0(\Omega))$ and $y(T)- y_d(T) \in H^1_0(\Omega)$. Hence, considering the time reversal $\Theta = T-t$ we can apply Theorem~\ref{thm:well_posedness_heat} to obtain the existence and uniqueness of a weak solution $p' \in W(0,T)$. Further, we can  evaluate the adjoint heat equation (cf.\ Theorem~\ref{thm:fonc_heat}) for every $t\in [0,T]$ and, in particular, for $t=T$ we obtain 
\begin{equation}
	\label{eq:adjoint_compatibility_condition}
	-p'(T) - \laplace p(T) = -p'(T) = y(T) - y_d(T),
\end{equation}
which, similarly to Section~\ref{ssec:forward_asymptotic}, corresponds to the terminal condition in \eqref{eq:adjoint_heat_derivative}. Therefore, by construction, the adjoint equation satisfies the compatibility condition \eqref{eq:adjoint_compatibility_condition} for $\nu=1$ (cf.~Theorem~\ref{thm:fonc_cattaneo}). 
\begin{Remark}
	As mentioned in Remark~\ref{rem:compatibility_adjoint}, the adjoint Cattaneo equation does not fulfill the adjoint compatibility condition naturally. In fact, when considering $\nu = 0$, we would obtain $p_\tau'(T,\cdot) = 0$ (cf.~Theorem~\ref{thm:fonc_cattaneo}). 
\end{Remark}

Now, we have everything at hand to prove the convergence of the optimal control and the optimal state.
\begin{Theorem}
	\label{thm:linear_convergence}
	Let $\nu=1$, $\lambda > 0$, $y_0 \in H^3(\Omega) \cap H^1_0(\Omega)$, $y_1 = \laplace y_0 + \bar{u} \in H^1_0(\Omega)$, $y_d \in H^1(0,T;H^1_0(\Omega))$ and let Assumption~\ref{ass:regularity} hold. We denote with $u_\tau^* \in U_\mathrm{ad}$ and $y_\tau^* \in Y(0,T)$ the optimal control and optimal state of problem \eqref{eq:ocp_cattaneo}, respectively, and with $u^* \in U_\mathrm{ad}$ and $y^* \in W(0,T)$ the optimal control and state of problem \eqref{eq:ocp_heat}, respectively. Then, there exists a constant $C= C(\lambda, y_0, y_d) > 0$, independent of $\tau$, such that
	\begin{equation*}
		\norm{u_\tau^* - u^*}{L^2(0,T;L^2(\Omega))} \leq \tau C(\lambda, y_0, y_d) \qquad \text{ and } \qquad \norm{y_\tau^* - y^*}{L^2(0,T;H^1_0(\Omega))} \leq \tau C(\lambda, y_0, y_d).
	\end{equation*}
	In particular, the optimal controls and the optimal states converge linearly as $\tau \to 0$.
\end{Theorem}
\begin{proof}	
	Using $p_\tau^* - p^*$ as test function in the weak formulation of the Cattaneo equation yields
	\begin{equation}
		\label{eq:tested_state}
		\tau \int_0^T \left((y_\tau^*)'', p_\tau^* - p^*\right) + \left( (y_\tau^* - y^*)', p_\tau^* - p^* \right) + a\left[ y_\tau^* - y^*, p_\tau^* - p^* \right] \dx{t} = \int_{0}^{T} \left( u_\tau^* - u^*, p_\tau^* - p^* \right) \dx{t},
	\end{equation}
	and, similarly, taking $y_\tau^* - y^*$ as test function for the adjoint Cattaneo equation gives
	\begin{equation}
		\label{eq:tested_adjoint}
		\tau \int_{0}^{T} \left( (p_\tau^*)'', y_\tau^* - y^* \right) - \left( (p_\tau^* - p^*)', y_\tau^* - y^* \right) + a\left[ p_\tau^* - p^*, y_\tau^* - y^* \right] \dx{t} = \int_{0}^{T} \left( y_\tau^* - y^*, y_\tau^* - y^* \right) \dx{t}.
	\end{equation}
	Subtracting \eqref{eq:tested_state} from \eqref{eq:tested_adjoint} results in
	\begin{equation}
		\label{eq:subtracted_equations}
		\begin{aligned}
			&\tau \int_{0}^{T} \left( (p_\tau^*)'', y_\tau^* - y^* \right) - \left( (y_\tau^*)'', p_\tau^* - p^* \right) - \frac{d}{dt} \left( y_\tau^* - y^*, p_\tau^* - p^* \right) \dx{t} \\
			=\ &\norm{y_\tau^* - y^*}{L^2(0,T;L^2(\Omega))}^2 - \inner{u_\tau^* - u^*}{p_\tau^* - p^*}{L^2(0,T;L^2(\Omega))},
		\end{aligned}
	\end{equation}
	due to the symmetry of the bilinear form $a$.
	
	Note that the set of admissible controls for problems \eqref{eq:ocp_cattaneo} and \eqref{eq:ocp_heat} coincide. In fact, we are allowed to use the optimality conditions from Theorems~\ref{thm:fonc_cattaneo} and~\ref{thm:fonc_heat} to obtain
	\begin{equation*}
		-\inner{p_\tau^* + \lambda u_\tau^*}{u_\tau^* - u^*}{L^2(0,T;L^2(\Omega))} \geq 0 \quad \text{ and } \quad \inner{p^* + \lambda u^*}{u_\tau^* - u^*}{L^2(0,T;L^2(\Omega))} \geq 0.
	\end{equation*}
	Adding these inequalities yields
	\begin{equation}
		\label{eq:estimate_adjoint_control}
		- \inner{p_\tau^* - p^*}{u_\tau^* - u^*}{L^2(0,T;L^2(\Omega))} \geq \lambda \norm{u_\tau^* - u^*}{L^2(0,T;L^2(\Omega))}^2.
	\end{equation}
	Now, using the above estimate \eqref{eq:estimate_adjoint_control} in \eqref{eq:subtracted_equations} yields the following inequality
	\begin{equation}
		\label{eq:pre_estimate}
		\begin{aligned}
			&\norm{y_\tau^* - y^*}{L^2(0,T;L^2(\Omega))}^2 + \lambda \norm{u_\tau^* - u^*}{L^2(0,T;L^2(\Omega))}^2 \\
			\leq\ &\tau \int_{0}^{T} \left( (p_\tau^*)'', y_\tau^* - y^* \right) - \left( (y_\tau^*)'', p_\tau^* - p^* \right) - \frac{d}{dt} \left( y_\tau^* - y^*, p_\tau^* - p^* \right) \dx{t}.
		\end{aligned}
	\end{equation}
	
	In the following, we estimate the terms on the right-hand side of \eqref{eq:pre_estimate}. The last term vanishes as
	\begin{equation*}
		\int_0^T \frac{d}{dt} \left( y_\tau^* - y^*, p_\tau^* - p^* \right) \dx{t} = \left( y_\tau^*(T) - y^*(T) \right) \left( p_\tau^*(T) - p^*(T) \right) - \left( y_\tau^*(0) - y^*(0) \right) \left( p_\tau^*(0) - p^*(0) \right) = 0,
	\end{equation*}
	since $y_\tau^*(0) = y^*(0) = y_0$ and $p_\tau^*(T) = p^*(T) = 0$.
	For the estimation of the first and second term, we have the following inequalities due to Theorem~\ref{thm:linear_convergence_state}
	\begin{equation}
		\label{eq:estimate_state}
		\begin{aligned}
			\norm{y_\tau^* - y^*}{L^2(0,T;H^1_0(\Omega))} &\leq C \norm{u_\tau^* - u^*}{L^2(0,T;L^2(\Omega))} + \tau C \left( M + \norm{y_0}{H^3(\Omega)} \right) \\
			&= C \norm{u_\tau^* - u^*}{L^2(0,T;L^2(\Omega))} + \tau C(y_0),
		\end{aligned}
	\end{equation}
	as well as
	\begin{equation*}
		\begin{aligned}
			\norm{p_\tau^* - p^*}{L^2(0,T;H^1_0(\Omega))} &\leq C \norm{y_\tau^* - y^*}{L^2(0,T;H^1_0(\Omega))} + \tau C \left( \norm{y_\tau^*}{H^1(0,T;H^1_0(\Omega))} + \norm{y_d}{H^1(0,T;H^1_0(\Omega))} \right) \\
			&\leq C \norm{u_\tau^* - u^*}{L^2(0,T;L^2(\Omega))} + \tau C \left( M + \norm{y_0}{H^3(\Omega)} + \norm{y_d}{H^1(0,T;H^1_0(\Omega))} \right) \\
			&= C \norm{u_\tau^* - u^*}{L^2(0,T;L^2(\Omega))} + \tau C(y_0, y_d),
		\end{aligned}
	\end{equation*}
	where we have used the time reversal $\theta = T - t$ for the adjoint equation, the embedding $L^2(0,T;H^1_0(\Omega)) \embedding L^2(0,T;L^2(\Omega))$, and the triangle inequality. Additionally, due to \eqref{eq:energy_estimate_complete} we have the estimates
	\begin{equation*}
		\norm{(y_\tau^*)''}{L^2(0,T;L^2(\Omega))} \leq C \left( \norm{u_\tau^*}{H^1(0,T;H^1(\Omega))} + \norm{y_0}{H^3(\Omega)} \right) \leq C \left( M + \norm{y_0}{H^3(\Omega)} \right) = C(y_0)
	\end{equation*}
	and
	\begin{equation*}
		\begin{aligned}
			\norm{(p_\tau^*)''}{L^2(0,T;L^2(\Omega))} &\leq C \left( \norm{y_\tau^*}{H^1(0,T;H^1_0(\Omega))} + \norm{y_d}{H^1(0,T;H^1_0(\Omega))} \right) \\
			&\leq C \left( M + \norm{y_0}{H^3(\Omega)} + \norm{y_d}{H^1(0,T;H^1_0(\Omega))} \right) = C(y_0, y_d).
		\end{aligned}
	\end{equation*}
	Therefore, we can estimate the first term of \eqref{eq:pre_estimate} as follows
	\begin{equation*}
		\begin{aligned}
			\abs{ \int_0^T \left( (p_\tau^*)'', y_\tau^* - y^* \right) \dx{t} } &\leq \norm{(p_\tau^*)''}{L^2(0,T;L^2(\Omega))} \norm{y_\tau^* - y^*}{L^2(0,T;H^1_0(\Omega))} \\
			&\leq C(y_0, y_d) \left( \norm{u_\tau^* - u^*}{L^2(0,T;L^2(\Omega))} + \tau C(y_0) \right).
		\end{aligned}
	\end{equation*}
	For the second term on the right-hand side of \eqref{eq:pre_estimate}, we can proceed analogously and obtain
	\begin{equation*}
		\begin{aligned}
			\abs{ \int_0^T \left( (y_\tau^*)'', p_\tau^* - p^* \right) \dx{t}} &\leq \norm{(y_\tau^*)''}{L^2(0,T;L^2(\Omega))} \norm{p_\tau^* - p^*}{L^2(0,T;H^1_0(\Omega))} \\
			&\leq C(y_0) \left( \norm{u_\tau^* - u^*}{L^2(0,T;L^2(\Omega))} + \tau C(y_0, y_d) \right).
		\end{aligned}
	\end{equation*}
	
	Using the previous estimates in \eqref{eq:pre_estimate}, we obtain the following inequality
	\begin{equation}
		\label{eq:linear_quadratic_estimate}
		\lambda \norm{u_\tau^* - u^*}{L^2(0,T;L^2(\Omega))}^2 \leq C(y_0, y_d) \left( \tau \norm{u_\tau^* - u^*}{L^2(0,T;L^2(\Omega))}  + \tau^2 \right).
	\end{equation}
	We observe, that \eqref{eq:linear_quadratic_estimate} is of the form $\lambda x^2 \leq C(y_0, y_d) \left( \tau x + \tau^2 \right)$ with $x= \norm{u_\tau^* - u^*}{L^2(0,T;L^2(\Omega))}$. We can use this to estimate
	\begin{equation*}
		\begin{aligned}
			\frac{\lambda x^2}{C(y_0, y_d)} \leq \tau x + \tau^2 \quad &\Rightarrow \quad \frac{\lambda x^2}{C(y_0, y_d)} \leq \left( x+\tau \right)^2 - x^2 \quad \Rightarrow  \quad \sqrt{ 1 + \frac{\lambda}{C(y_0, y_d)}} x \leq x+\tau \\
			&\Rightarrow \left( \sqrt{ 1 + \frac{\lambda}{C(y_0, y_d)}} - 1 \right) x \leq \tau,
		\end{aligned}
	\end{equation*}
	due to the non-negativity of $x$ and $\tau$. Since $\sqrt{ 1 + \frac{\lambda}{C(y_0, y_d)}} - 1 > 0$, there exists a constant $C(\lambda, y_0, y_d) > 0$ so that
	\begin{equation}
		\label{eq:convergence_controls}
		\norm{u_\tau^* - u^*}{L^2(0,T;L^2(\Omega))} \leq \tau C(\lambda, y_0, y_d),
	\end{equation}
	which gives the linear convergence of the optimal controls. Further, inserting \eqref{eq:convergence_controls} into \eqref{eq:estimate_state} yields the linear convergence of the optimal states, i.e., there exists some constant $C(\lambda, y_0, y_d) > 0$ so that
	\begin{equation*}
		\norm{y_\tau^* - y^*}{L^2(0,T;H^1_0(\Omega))} \leq \tau C(\lambda, y_0, y_d),
	\end{equation*}
	which completes the proof.
\end{proof}

\section{Numerical Results}
\newcommand{\scaleimgs}{1.0}
\newcommand{\tablegray}{gray!20}
\label{sec:numerical results}

We provide numerical results for the limit $\tau \to 0$ that underline our previous analysis. For the simulation of the Cattaneo equation we use the Newmark beta method (cf. \cite{Newmark1959Method}) with parameters $\beta = \gamma = \nicefrac{1}{2}$ as semi-discretization in time and for the simulation of the heat equation we use the Crank-Nicolson method for the time discretization. In both cases, the resulting sequence of PDEs is solved in Python with FEniCS \cite{Logg2012Automated,Alnes2015FEniCS}. We use the domain $\Omega = (0,1)^2$ and the time horizon $T=1.0$ for all simulations. 
%
%In order to find an appropriate time discretization we did a convergence analysis which revealed that the numerical solution showed no significant differences for a time step $\Delta t \approx \tau $. 
A convergence analysis revealed that an appropriate time discretization is given by $\Delta t \approx \tau $. Since the smallest $\tau$ we consider in this analysis is $\tau=1e^{-4}$, we choose a step size of $\Delta t = 1e^{-4}$ for all computations. For the spatial discretization we use a mesh consisting of $2601$ nodes in a uniform grid of $5000$ triangles and we use piecewise linear Lagrange finite elements. 
We denote the optimal controls and optimal states of the heat and Cattaneo equation with $u^*$ and $u_\tau^*$ as well as $y^*$ and $y_\tau^*$, respectively.

\begin{Remark}
	For the sake of brevity, we only consider the asymptotic behavior of the Cattaneo equation in the context of optimal control problems numerically. A detailed numerical investigation of the forward problem can be found in \cite{Blauth2018Optimal}.
\end{Remark}

For the numerical optimization with the Cattaneo and heat equation we only consider the case $\nu = 1$, as this is the main focus of the paper. For a discussion of the case $\nu = 0$ we refer the reader to \cite{Blauth2018Optimal}. %the second optimal control problem for the Cattaneo equation, i.e., we investigate the convergence of the optimal controls and states of \eqref{eq:oc_cattaneo} to the ones of \eqref{eq:distributed_control_heat}. 

We choose homogeneous initial conditions $y_0 = 0$ and $y_1 = \laplace y_0 + \bar u$ to satisfy the compatibility condition \eqref{eq:compatibility_condition}. In contrast to our theoretical investigation in Section~\ref{sec:asymptotic analysis} we choose the set of admissible controls $U_\mathrm{ad} = H^1(0,T;H^1(\Omega))$, i.e., we do not assume that $U_\mathrm{ad}$ is bounded for the numerical optimization. As desired state $y_d$ we choose the function
\begin{equation*}
	y_d(t,x) = \exp\left(-20 \left( \left( x_1 - \delta_{1}(t) \right)^2 + \left( x_2 - \delta_{2}(t) \right)^2 \right)\right) \qquad\text{with}\qquad 
	\begin{bmatrix}
		\delta_1(t) \\
		\delta_2(t)
	\end{bmatrix}
	= 
	\begin{bmatrix}
		\nicefrac{1}{2} + \nicefrac{1}{4} \cdot\cos(2\pi t) \\
		\nicefrac{1}{2} + \nicefrac{1}{4} \cdot\sin(2\pi t)
	\end{bmatrix},
\end{equation*}
which describes a Gaussian pulse centred at $[\delta_1(t), \delta_2(t)]^T$ that moves counterclockwise along a circle with midpoint $[\nicefrac{1}{2}, \nicefrac{1}{2}]^T$ and radius $\nicefrac{1}{4}$, starting at $[\nicefrac{3}{4},\nicefrac{1}{2}]^T$ for $t=0$. Note that the desired state does not have zero boundary values, in contrast to the assumption of the theoretical analysis. \\
%
%In order 
To solve this optimization problem numerically, we use the gradient descent method with exact line search since we have no control constraints for the model problem (cf. \cite{Hinze2009Optimization} or \cite{Schwedes2017Mesh}). The gradient descent method is initialized with $u^0= 0$ in $[0,T]\times \Omega$, the relative stopping tolerance is \num{1e-4}. \\

\begin{table}[!b]
	\centering
	\caption{Absolute and relative difference of the optimal controls and states for $\lambda = 1.0$.}
	\label{tab:convergence_table}
	\begin{subtable}{0.5\textwidth}
		\centering
		\caption{$\norm{u_\tau^* - u^*}{L^2(0,T;L^2(\Omega))}$}
		\rowcolors{2}{\tablegray}{white}
		\begin{tabular}{l l S l}
			\toprule
			$\tau$ & absolute & {relative [\%]} & order \\
			\midrule
			\num{1e-0} & \num{2.71e-2} & 288.1 & \\
			\num{1e-0.5} & \num{1.03e-2} & 109.4 & \num{0.84} \\
			\num{1e-1} & \num{3.07e-3} & 32.6 & \num{1.05} \\
			\num{1e-1.5} & \num{8.88e-4} & 9.44 & \num{1.08} \\
			\num{1e-2} & \num{2.61e-4} & 2.78 & \num{1.06} \\
			\num{1e-2.5} & \num{7.96e-5} & 0.85 & \num{1.03} \\
			\num{1e-3} & \num{2.48e-5} & 0.26 & \num{1.01} \\
			\num{1e-3.5} & \num{7.8e-6} & 0.083 & \num{1.0} \\
			\num{1e-4} & \num{2.46e-6} & 0.026 & \num{1.0} \\
			\bottomrule
		\end{tabular}
	\end{subtable}%
	\begin{subtable}{0.5\textwidth}
		\centering
		\caption{$\norm{y_\tau^* - y^*}{L^2(0,T;H^1_0(\Omega))}$}
		\rowcolors{2}{\tablegray}{white}
		\begin{tabular}{l l S l}
			\toprule
			$\tau$ & absolute & {relative [\%]} & order \\
			\midrule
			\num{1e-0} & \num{1.59e-2} & 789.5 & \\
			\num{1e-0.5} & \num{4.55e-3} & 226.43 & \num{1.08} \\
			\num{1e-1} & \num{9.76e-4} & 48.54 & \num{1.34} \\
			\num{1e-1.5} & \num{2.69e-4} & 13.39 & \num{1.12} \\
			\num{1e-2} & \num{7.90e-5} & 3.93 & \num{1.06} \\
			\num{1e-2.5} & \num{2.43e-5} & 1.21 & \num{1.02} \\
			\num{1e-3} & \num{7.60e-6} & 0.38 & \num{1.0} \\
			\num{1e-3.5} & \num{2.40e-6} & 0.12 & \num{1.0} \\
			\num{1e-4} & \num{7.58e-6} & 0.038 & \num{1.0} \\
			\bottomrule
		\end{tabular}
	\end{subtable}%
\end{table}

First, we investigate the convergence for $\lambda = 1.0$. The results are shown in Table~\ref{tab:convergence_table}, where we show the absolute and relative differences of the optimal controls and states as well as the numerical order of convergence. Additionally, these results are also visualized in Figures~\ref{fig:ocp_control_lambdas} and~\ref{fig:ocp_state_lambdas}, where the absolute differences are shown in a log-log plot. 

\begin{figure}[!b]
	\centering
	\includegraphics[scale=\scaleimgs]{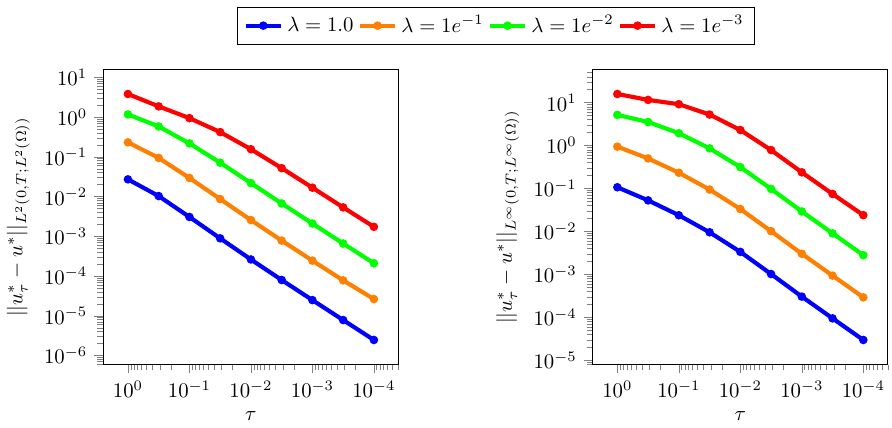}
	\caption{Convergence rates for the optimal controls of the Cattaneo equation.}
	\label{fig:ocp_control_lambdas}
\end{figure}

\begin{figure}[!b]
	\centering
	\includegraphics[scale=\scaleimgs]{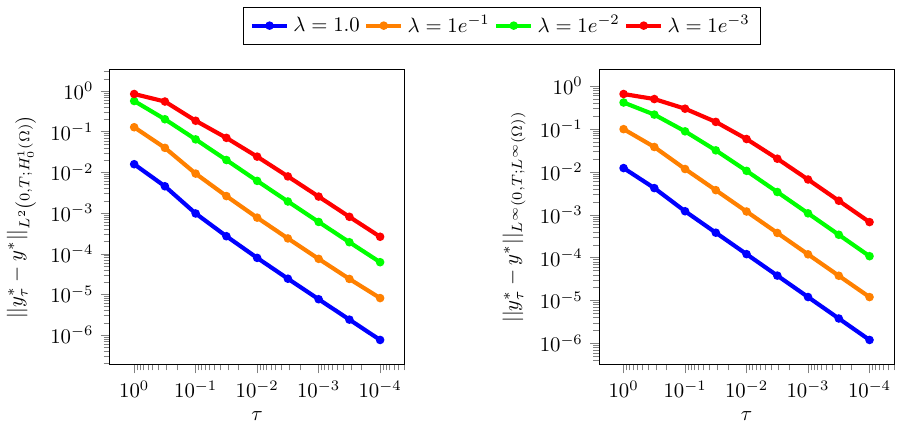}
	\caption{Convergence rates for the optimal states of the Cattaneo equation.}
	\label{fig:ocp_state_lambdas}
\end{figure}

The numerical results displayed there show that we indeed have convergence of both the optimal states and the optimal controls in all considered norms, which we have proved theoretically in Theorem~\ref{thm:linear_convergence} under slightly stronger assumptions. The numerical results confirm our theoretical findings: We have a linear convergence in the $L^2(0,T;H^1_0(\Omega))$ norm for the optimal state as well as a linear convergence in the $L^2(0,T;L^2(\Omega))$ norm for the optimal control. We observe that the convergence is weaker in the $L^\infty(0,T;L^\infty(\Omega))$ norm, particularly in the beginning. However, the results again suggest that we might have a linear convergence in this norm as $\tau \to 0$.

Let us now consider the optimal control problem in the limit $\lambda \to 0$. The convergence rates for different $\lambda$ are depicted in Figures~\ref{fig:ocp_control_lambdas} and Figure~\ref{fig:ocp_state_lambdas}. The results show that the numerical optimization behaves well for $\lambda \to 0$, as we still get a very similar qualitative convergence behaviour for all considered values of $\lambda$. Again, the numerical results confirm the theoretical analysis, since we observe that the convergence of the optimal state is linear in the $L^2(0,T;H^1_0(\Omega))$ norm and that the convergence of the optimal control is linear in the $L^2(0,T;L^2(\Omega))$ even for $\lambda \to 0$.

\section{Conclusion and Outlook}
\label{sec:conclusion}
We investigated the (forward) Cattaneo equation and the optimal control problem constrained by the Cattaneo equation. Our main focus was the asymptotic analysis of the limit $\tau\to 0$. For this, we recalled results concerning the well-posedness of the Cattaneo equation and heat equation for homogeneous Dirichlet boundary data.% which were chosen in order to simplify our notation and computations and briefly discussed the optimal control problem for the Cattaneo and the heat equation, where we stated the existence and uniqueness of an optimal control as well as the first order optimality conditions.

With these results, we examined the asymptotic analysis of both the forward problem and the optimal control problem for $\tau\to 0$. First, we summarized the results obtained in \cite{Blauth2018Optimal}, where the asymptotic behaviour was already investigated without proving a convergence rate. We improved these results under additional assumptions. Particularly, we proved the linear convergence of solutions of the Cattaneo equation to the ones of the heat equation. Moreover, we proved linear convergence of the solutions to the optimal control problem as $\tau \to 0$.
Finally, we provided numerical results for the optimal control problem constrained by the Cattaneo equation in the context of a vanishing delay time that confirmed the theoretical analysis.

Further theoretical and numerical analysis of the Cattaneo model in the context of radiative heat transfer is of interest for medical applications such as laser-induced thermotherapy or microwave ablation for the treatment of tumors. For such problems, the Cattaneo model for heat transfer has to be coupled with nonlinear equations describing radiative effects, which makes the analysis more challenging. 

%Further analysis and numerics involving the investigation of the Cattaneo model with non-linear terms, possibly coupled with other models such as radiative heat transfer is planned. %theoretically and numerically. Of practical interest would be the investigation of the Cattaneo model with non-linear terms, possibly coupled with other models such as, e.g., radiative heat transfer. 
%A first step in this direction is the coupling of the Cattaneo model with %Therefore, one could investigate the Cattaneo model coupled with 
%simplified $P_N$-approximations which are of practical interest in medical applications.

\section*{Acknowledgments}
Ren\'e Pinnau and Matthias Andres are grateful for the support of the German Federal Ministry of Education and Research (BMBF) grant no. 05M16UKE.

\appendix
\gdef\thesection{\Alph{section}}
\makeatletter
\renewcommand\@seccntformat[1]{\appendixname\ \csname the#1\endcsname.\hspace{0.5em}}
\makeatother

\section{Proofs of Section~\ref{sec:asymptotic analysis}}
\label{app:proofs}

In this section we provide the proofs of the results from Section~\ref{sec:asymptotic analysis}, which are based on the proofs given in \cite[Chapter~5]{Blauth2018Optimal}.

\subsection{Proof of Lemma~\ref{lem:boundedness}}
\label{proof:boundedness}

We use the Faedo-Galerkin method to obtain the energy estimates from Lemma~\ref{lem:boundedness}. In particular, we choose a countable set of linearly independent functions $w_k \in H^1_0(\Omega)$ so that the span of $\Set{w_k | k\in \mathbb{N}}$ is dense in $H^1_0(\Omega)$. Such a set exists since $H^1_0(\Omega)$ is separable. Let $m \in \mathbb{N}_{>0}$ and construct a function $y_m$ of the form
\begin{equation}
	\label{eq:faedo_galerkin}
	y_m(t) = \sum\limits_{k=1}^m \varphi_{mk}(t) w_k, \quad \varphi_{mk} \in H^2(0,T).
\end{equation}
In \cite{Blauth2018Optimal}, it is shown that there exists a unique $y_m$ of the form \eqref{eq:faedo_galerkin} which satisfies
\begin{equation}
	\label{eq:approximate_cattaneo}
	\begin{aligned}
		\left( \tau y_m'', w_k \right) + \left( y_m', w_k \right) + a[y_m, w_k] = \left( u, w_k \right), \\
		\varphi_{mk} = \beta_{mk}, \quad \varphi_{mk}' = \gamma_{mk},
	\end{aligned}
\end{equation}
for a.e. $t\in [0,T]$ and $k=1, \dots, m$, where the coefficients $\beta_{mk}$ and $\gamma_{mk}$ are chosen such that
\begin{equation*}
	\begin{aligned}
		y_m(0) = \sum\limits_{k=1}^m \beta_{mk} w_k \to y_0 \text{ in } H^1_0(\Omega) \qquad \text{ and } \qquad y_m'(0) = \sum\limits_{k=1}^m \gamma_{mk} w_k \to y_1 \text{ in } L^2(\Omega).
	\end{aligned}
\end{equation*}
For a detailed discussion we refer the reader to \cite[Chapter~3]{Blauth2018Optimal}.

\begin{Lemma}
	\label{lem:boundedness_y_prime}
	Let $y_m$ be the approximate solution of \eqref{eq:approximate_cattaneo} described above. Then, there exists a constant $C>0$ independent of $\tau$ such that we have
	\begin{equation*}
		\norm{y_m'}{L^2(0,T;L^2(\Omega))}^2 \leq C \left(  \norm{u}{L^2(0,T;L^2(\Omega))}^2 + \norm{y_0}{H^1_0(\Omega)}^2 +  \norm{y_1}{L^2(\Omega)}^2\right).
	\end{equation*}
\end{Lemma}
\begin{proof}
	We multiply \eqref{eq:approximate_cattaneo} with $y_m'$ and integrate this over $[0,T]$ to obtain
	\begin{equation*}
		\begin{aligned}
			\norm{y_m'}{L^2(0,T;L^2(\Omega))}^2 &\leq \int\limits_{0}^{T} \left( u(t), y_m'(t) \right) \dx{t} + \frac{\tau}{2} \norm{y_m'(0)}{L^2(\Omega)}^2 + \frac{1}{2} a[y_m(0),y_m(0)] \\
			&\leq \frac{1}{2} \norm{u}{L^2(0,T;L^2(\Omega))}^2 + \frac{1}{2} \norm{y_m'}{L^2(0,T;L^2(\Omega))}^2 + C \left( \norm{y_0}{H^1_0}^2 + \norm{y_1}{L^2}^2 \right),
		\end{aligned}
	\end{equation*}
	where we have used H\"older's and Young's inequalities. Subtracting $\nicefrac{1}{2} \norm{y_m'}{L^2(0,T;L^2(\Omega))}^2$ yields the desired estimate.
\end{proof}

\begin{Lemma}
	\label{lem:poincare_time}
	Let $y\in H^1(0,T;L^2(\Omega))$. Then, there exists a constant $C(T)>0$ only depending on $T$ such that 
	$$ \norm{y}{L^2(0,T;L^2(\Omega))}^2 \leq C(T) \left( \norm{y'}{L^2(0,T;L^2(\Omega))}^2 + \norm{y(0)}{L^2(\Omega)}^2 \right).
	$$
\end{Lemma} 
\begin{proof}
	Note that $C^1([0,T];L^2(\Omega))$ is dense in $H^1(0,T;L^2(\Omega))$, so it suffices to prove the inequality for $y\in C^1([0,T];L^2(\Omega))$. For such a $y$ we observe
	\begin{equation*}
		\begin{aligned}
		\norm{y(t) - y(0) }{L^2(\Omega)}^2 = \norm{\int\limits_{0}^{t} y'(s) \dx{s}}{L^2(\Omega)}^2 \leq \left( \int\limits_{0}^{t} \norm{y'(s)}{L^2(\Omega)} \dx{s}\right)^2
		\leq t \int\limits_{0}^{t} \norm{y'(s)}{L^2(\Omega)}^2 \dx{s} \leq T \norm{y'}{L^2(0,T;L^2(\Omega))}^2,
		\end{aligned}
	\end{equation*}
	where we have used the Cauchy-Schwarz inequality in the second estimation. Integration over $[0,T]$ yields
	\begin{equation*}
		\norm{y-y(0)}{L^2(0,T;L^2(\Omega))}^2 = \int\limits_{0}^{T} \norm{y(t)-y(0)}{L^2(\Omega)}^2 \dx{t} \leq T \int\limits_{0}^{T} \norm{y'}{L^2(0,T;L^2(\Omega))}^2 \dx{t} = T^2 \norm{y'}{L^2(0,T;L^2(\Omega))}^2.
	\end{equation*}
	Finally, we employ the triangle inequality to find
	\begin{equation*}
		\norm{y}{L^2(0,T;L^2(\Omega))} \leq \norm{y - y(0)}{L^2(0,T;L^2(\Omega))} + \norm{y(0)}{L^2(0,T;L^2(\Omega))} \leq T \left( \norm{y'}{L^2(0,T;L^2(\Omega))} + \norm{y(0)}{L^2(\Omega)}^2 \right). \qedhere
	\end{equation*}
\end{proof}

We know that the approximate solution $y_m$ is in $H^2(0,T;H^1_0(\Omega)) \embedding H^1(0,T;L^2(\Omega))$ and, thus, we apply Lemma~\ref{lem:boundedness_y_prime} and Lemma~\ref{lem:poincare_time} to find
\begin{equation}
	\label{eq:estimate_y_l2}
	\begin{aligned}
		\norm{y_m}{L^2(0,T;L^2(\Omega))}^2 \leq\ &C(T)\left( \norm{y_m'}{L^2(0,T;L^2(\Omega))}^2 + \norm{y_m(0)}{L^2(\Omega)}^2 \right)\\
		\leq\ &C(T) \left( \norm{u}{L^2(0,T;L^2(\Omega))}^2 + \norm{y_0}{H^1_0(\Omega)}^2 +  \norm{y_1}{L^2(\Omega)}^2 \right),
	\end{aligned}
\end{equation}
with a constant $C(T) >0$ independent of $\tau$. 
\begin{Lemma}
	\label{lem:boundedness_y}
	There exists a constant $C>0$ independent of $\tau$ such that
	$$ \norm{y_m}{L^2(0,T;H^1_0(\Omega))}^2 \leq C \left( \norm{u}{L^2(0,T;L^2(\Omega))}^2 + \norm{y_0}{H^1_0(\Omega)}^2 + \norm{y_1}{L^2(\Omega)}^2 \right).
	$$
\end{Lemma}
\begin{proof}
	We multiply \eqref{eq:approximate_cattaneo} with $y_m$ and integrate this over $[0,T]$ to obtain
%	\begin{equation}
%		\label{eq:test_y_m}
%		\tau \left( y_m'', y_m \right) + \left( y_m', y_m \right) + a[y_m,y_m] = \left( u, y_m \right) \quad \text{ for a.e. } t\in [0,T].
%	\end{equation}
%	We observe that
%	\begin{equation*}
%	\tau \left( y_m'', y_m \right) = \tau \frac{d}{dt} \left( y_m', y_m\right) - \tau \norm{y_m'}{L^2(\Omega)}^2 \quad \text{ for a.e. } t\in [0,T].
%	\end{equation*}
%	Using this, we rewrite \eqref{eq:test_y_m} as
%	\begin{equation*}
%	\tau \frac{d}{dt} \left( y_m'(t), y_m(t)\right) + \frac{1}{2} \frac{d}{dt} \norm{y_m(t)}{L^2(\Omega)}^2 + a[y_m(t),y_m(t)] = \left( u(t), y_m(t) \right) + \tau \norm{y_m'(t)}{L^2(\Omega)}^2 \quad \text{ for a.e. } t\in [0,T].
%	\end{equation*}
%	Integrating this over $[0,T]$ yields
	\begin{equation*}
		\begin{aligned}
			\int\limits_{0}^{T} a[y_m(t),y_m(t)] \dx{t} &= \int\limits_{0}^{T} \left( u(t), y_m(t) \right) + \tau \norm{y_m'(t)}{L^2(\Omega)}^2\dx{t} - \tau \left( y_m'(T), y_m(T) \right)  \\
			&\quad + \tau \left( y_m'(0), y_m(0) \right) - \frac{1}{2} \norm{y_m(T)}{L^2(\Omega)}^2 + \frac{1}{2} \norm{y_m(0)}{L^2(\Omega)}^2.
		\end{aligned}
	\end{equation*}
	Using the coercivity of $a$ with the Cauchy-Schwarz and Young's inequality reveals
	\begin{equation}
		\label{eq:ineq_y}
		\begin{aligned}
			\norm{y_m}{L^2(0,T;H^1_0(\Omega))}^2 &\leq C \left( \norm{u}{L^2(0,T;L^2(\Omega))}^2 + \norm{y_m}{L^2(0,T;L^2(\Omega))}^2 + \norm{y_m'}{L^2(0,T;L^2(\Omega))}^2 + \frac{\tau}{2} \norm{y_m'(T)}{L^2(\Omega)}^2 \right. \\
			&\quad\qquad \left. + \frac{\tau}{2} \norm{y_m(T)}{L^2(\Omega)}^2 + \norm{y_1}{L^2(\Omega)}^2 + \norm{y_0}{H^1_0(\Omega)}^2 \right).
		\end{aligned}
	\end{equation}
	
	We estimate the terms on the right-hand side of \eqref{eq:ineq_y}. The second and third term can be estimated with \eqref{eq:estimate_y_l2} and Lemma~\ref{lem:boundedness_y_prime}. We estimate the term $\frac{\tau}{2}\norm{y_m'(T)}{L^2(\Omega)}^2$ as follows. We have
	\begin{equation}
		\label{eq:estimate_y_prime_at_zero}
		\frac{\tau}{2} \norm{y_m'(T)}{L^2(\Omega)}^2 = \frac{\tau}{2} \norm{y_m'(0)}{L^2(\Omega)}^2 + \int\limits_{0}^{T} \tau \left( y_m''(t), y_m'(t) \right) \dx{t}.
	\end{equation}
	Using the weak form of the Cattaneo equation \eqref{eq:weak_cattaneo} allows us to rewrite this as
	\begin{equation*}
		\frac{\tau}{2} \norm{y_m'(T)}{L^2(\Omega)}^2 = \frac{\tau}{2} \norm{y_m'(0)}{L^2(\Omega)}^2 + \int\limits_{0}^{T} \left( u(t), y_m'(t) \right) - \left( y_m'(t), y_m'(t)\right) - \frac{1}{2} \frac{d}{dt} a[y_m(t),y_m(t)] \dx{t}.
	\end{equation*}
	Using the coercivity of $a$ together with the Cauchy-Schwarz and Young's inequality as well as Lemma~\ref{lem:boundedness_y_prime} gives the estimate
	\begin{equation*}
		\frac{\tau}{2} \norm{y_m'(T)}{L^2(\Omega)}^2 \leq C\left( \norm{u}{L^2(0,T;L^2(\Omega))}^2 + \norm{y_0}{H^1_0(\Omega)}^2 + \norm{y_1}{L^2(\Omega)}^2 \right),
	\end{equation*}
	independently of $\tau$.
	
	Let us now investigate the term $\frac{\tau}{2}\norm{y_m(T)}{L^2(\Omega)}^2$. Similarly to before, we write this as 
	\begin{equation}
		\label{eq:help_y_at_T}
		\frac{\tau}{2} \norm{y_m(T)}{L^2(\Omega)}^2 = \frac{\tau}{2} \norm{y_m(0)}{L^2(\Omega)}^2 + \int\limits_{0}^{T} \tau \left( y_m'(t), y_m(t) \right) \dx{t}.
	\end{equation}
	Estimating the right-hand side of \eqref{eq:help_y_at_T} with the Cauchy-Schwarz and Young's inequality as well as Lemma~\ref{lem:boundedness_y_prime} and \eqref{eq:estimate_y_l2} gives
	\begin{equation*}
		\frac{\tau}{2} \norm{y_m(T)}{L^2(\Omega)}^2 \leq C\left( \norm{u}{L^2(0,T;L^2(\Omega))}^2 + \norm{y_0}{H^1_0(\Omega)}^2 + \norm{y_1}{L^2(\Omega)} \right),
	\end{equation*}
	also independently of $\tau$.
	
	Using all of the estimates above in \eqref{eq:ineq_y} we obtain the desired estimate
	$$ \norm{y_m}{L^2(0,T;H^1_0(\Omega))}^2 \leq C \left( \norm{u}{L^2(0,T;L^2(\Omega))}^2 + \norm{y_0}{H^1_0(\Omega)}^2 + \norm{y_1}{L^2(\Omega)}^2 \right). \qedhere
	$$
\end{proof}

\begin{proof}[Proof of Lemma~\ref{lem:boundedness}]
	The energy estimate follows directly from Lemma~\ref{lem:boundedness_y_prime} and Lemma~\ref{lem:boundedness_y}, by using the density of the Faedo-Galerkin basis functions $w_k$ in $H^1_0(\Omega)$ and $L^2(\Omega)$, respectively.
\end{proof}

\subsection{Proof of Theorem~\ref{thm:convergence_rateless}}
\label{proof:convergence_rateless}

\begin{proof}[Proof of Theorem~\ref{thm:convergence_rateless}]
	We denote by $y_i$ the unique weak solution of the Cattaneo equation with $\tau = \tau_i$. Thanks to Lemma~\ref{lem:boundedness} we know that the sequences $(y_i)$ and $(y_i')$ are bounded independently of $\tau$ so that we can extract weakly convergent subsequences $(y_{i_j}) \subset (y_i)$ and $(y_{i_j}') \subset (y_i')$ with
	\begin{equation*}
		y_{i_j} \rightharpoonup y \quad \text{ in } L^2(0,T;H^1_0(\Omega)) \quad\qquad \text{ and } \quad\qquad y_{i_j}' \rightharpoonup y' \quad \text{ in } L^2(0,T;L^2(\Omega))
	\end{equation*}
	for some $y\in L^2(0,T;H^1_0(\Omega))$ with weak time derivative $y'\in L^2(0,T;L^2(\Omega))$ as $j\to\infty$. In the following, we show that $y$ is a weak solution of the heat equation. Therefore, we assume that the test function $v$ lies in $C^\infty_0((0,T);H^1_0(\Omega))$. For such a $v$, we apply integration by parts in \eqref{eq:weak_cattaneo} and observe
	\begin{equation}
		\label{eq:limit_cattaneo}
		\int\limits_{0}^{T} -\tau_{i_j} \left( y_{i_j}'(t), v'(t) \right) + \left( y_{i_j}'(t), v(t) \right) + a[y_{i_j}(t),v(t)] \dx{t} = \int\limits_{0}^{T} \left( u(t), v(t) \right) \dx{t}.
	\end{equation}
	From Lemma~\ref{lem:boundedness} we get that
%	\begin{equation*}
%	\norm{\tau_{i_j} y_{i_j}'}{L^2(0,T;L^2(\Omega))}^2 = \tau_{i_j}^2 \norm{y_{i_j}'}{L^2(0,T;L^2(\Omega))}^2 \leq \tau_{i_j}^2 C\left( \norm{u}{L^2(0,T;L^2(\Omega))}^2 + \norm{y_0}{H^1_0(\Omega)}^2 + \norm{y_1}{L^2(\Omega)}^2 \right),
%	\end{equation*}
%	where $C>0$ is independent of $\tau$. Due to this, we have that
	\begin{equation*}
	\tau_{i_j} \norm{y_{i_j}'}{L^2(0,T;L^2(\Omega))} = \norm{\tau_{i_j} y_{i_j}'}{L^2(0,T;L^2(\Omega))} \to 0
	\end{equation*}
	as $j\to \infty$ since $y_i'$ is bounded independently of $\tau$. Therefore, passing to the limit $j\to\infty$ in \eqref{eq:limit_cattaneo} yields
	\begin{equation}
	\label{eq:pass_to_limit}
	\int\limits_{0}^{T} \left( y'(t), v(t) \right) + a[y(t),v(t)] \dx{t} = \int\limits_{0}^{T} \left( u(t), v(t)\right) \dx{t},
	\end{equation}
	due to the weak convergence of $y_{i_j}$ and $y_{i_j}'$. Note that this also holds true for all $v\in L^2(0,T;H^1_0(\Omega))$ since $C^\infty_0((0,T);H^1_0(\Omega))$ is dense in this space. 

	Let us now choose $v\in C^1([0,T];H^1_0(\Omega))$ with $v(T) = 0$. Applying integration by parts in \eqref{eq:pass_to_limit} gives
	\begin{equation}
		\label{eq:initial_condition_1}
		\int\limits_{0}^{T} - \left( y(t), v'(t) \right) + a[y(t),v(t)] \dx{t} = \int\limits_{0}^{T} \left( u(t), v(t)\right) \dx{t} + \left( y(0), v(0)\right).
	\end{equation}
	On the other hand, we also apply integration by parts to \eqref{eq:weak_cattaneo} and observe
	\begin{equation}
		\label{eq:initial_condition_2}
		\begin{aligned}
			&\int\limits_{0}^{T} -\tau_{i_j} \left( y'_{i_j}(t), v'(t) \right) - \left( y_{i_j}(t), v'(t) \right) + a[y_{i_j}(t),v(t)] \dx{t} \\
			= & \int\limits_{0}^{T} \left( u(t), v(t) \right) \dx{t} + \tau_{i_j} \left( y_1, v(0) \right) + \left(y_0, v(0) \right),
		\end{aligned}
	\end{equation}
	since we have $y_{i_j}(0,\cdot) = y_0$ for all $j$. Analogously to before, we take the limit $j\to \infty$ in \eqref{eq:initial_condition_2} and observe
	\begin{equation}
	\label{eq:initial_condition_2_limit}
		\int\limits_{0}^{T} - \left( y(t), v'(t)\right) + a[y(t),v(t)] \dx{t} = \int\limits_{0}^{T} \left( u(t), v(t)\right) \dx{t} + \left( y_0, v(0)\right).
	\end{equation}
	Comparing equations \eqref{eq:initial_condition_1} and \eqref{eq:initial_condition_2_limit} we obtain
	\begin{equation*}
	\left( y_0, v(0)\right) = \left( y(0), v(0)\right),
	\end{equation*}
	and, therefore, directly $ y(0) = y_0$ as $v(0)$ is arbitrary.
	
	As the heat equation \eqref{eq:weak_heat} has a unique solution and due to the embedding $Y(0,T) \embedding W(0,T)$, we observe that the weak limit $y$ is indeed the weak solution of the heat equation. Due to the uniqueness of the limit, we even get the convergence of the entire sequence, so that
	\begin{equation*}
		y_i \rightharpoonup y \quad \text{ in } L^2(0,T;H^1_0(\Omega)) \quad\qquad \text{ and } \quad \qquad y_i' \rightharpoonup y' \quad \text{ in } L^2(0,T;L^2(\Omega)).
	\end{equation*}
	
	We even get the strong convergence thanks to the Aubin-Lions Lemma (see, e.g., \cite[Lemma~3.74]{Ruzicka2020Nichtlineare}). We apply this with $X_0 = H^1_0(\Omega)$, $X=L^2(\Omega)$, $X_1 = L^2(\Omega)$ and $p = q = 2$. Hence, we have 
	\begin{equation*}
		W_0 := \Set{y\in L^2(0,T;H^1_0(\Omega)) | y'\in L^2(0,T;L^2(\Omega))} = W(0,T)
	\end{equation*}
	and we observe that the sequence $(y_i)$ is in $W_0$. Moreover, the embedding $H^1_0(\Omega) \embedding\embedding L^2(\Omega)$ is compact (see, e.g., \cite{Evans2010Partial}) and the embedding $L^2 \embedding L^2$ is continuous. Hence, the Aubin-Lions Lemma gives the compact embedding $W_0\embedding\embedding L^2(0,T;L^2(\Omega))$. Due to the boundedness of $(y_i)$ we get a subsequence $(y_{i_l})\subset (y_i)$ with
	\begin{equation*}
		y_{i_l} \to y \text{ in } L^2(0,T;L^2(\Omega)) \text{ as } l\to\infty,
	\end{equation*}
	due to the uniqueness of the weak limit. As before, the uniqueness of the limit also implies the convergence of the entire sequence and, hence, it holds that
	\begin{equation*}
		y_i \to y \text{ in } L^2(0,T;L^2(\Omega)) \text{ as } i\to \infty,
	\end{equation*}
	which completes the proof.
\end{proof}

\subsection{Proof of Theorem~\ref{thm:convergence_ocp_rateless}}
\label{proof:convergence_ocp_rateless}

\begin{proof}[Proof of Theorem~\ref{thm:convergence_ocp_rateless}]
	We denote by $G_{c_i}$ the solution operator of the Cattaneo equation with $\tau = \tau_i$, i.e.,
	$$G_{c_i}\colon L^2(0,T;L^2(\Omega)) \to Y(0,T);\quad u\mapsto G_{c_i}(u) = y_i,
	$$
	where $y_i$ is the unique weak solution of \eqref{eq:weak_cattaneo} with $\tau = \tau_i$. We consider the reduced problem
	\begin{equation*}
	\min_{u\in U_\mathrm{ad}} \hat{J}_i(u) = J(G_{c_i}(u), u),
	\end{equation*}
	which has a unique minimizer $u_i^* \in U_\mathrm{ad}$ for every $i\in\mathbb{N}$ due to Theorem~\ref{thm:well_posedness_cattaneo}, and we denote by $y_i^*$ its corresponding optimal state, i.e., $y_i^* := G_{c_i}(u_i^*)$.

	For the heat equation we have the solution operator 
	\begin{equation*}
		G\colon U \to W(0,T); \quad u\mapsto G(u) = y,
	\end{equation*}
	where $y$ is the unique weak solution of \eqref{eq:weak_heat}. Again, we consider the reduced problem
	\begin{equation*}
	\min_{u\in U_\mathrm{ad}} \hat{J}(u) = J(G(u), u).
	\end{equation*}
	Due to Theorem~\ref{thm:optimality_heat} the above problem has a unique minimizer $u^*\in U_\mathrm{ad}$ and we define $y^* := G(u^*)$, i.e., $y^*$ as the optimal state of the heat equation.
	
	In the following, we first show the convergence $u_i^* \rightharpoonup u^*$ in $L^2(0,T;L^2(\Omega))$ as well as $y_i^* \rightharpoonup y^*$ in $L^2(0,T;H^1_0(\Omega))$. 
	Since $u_i^*$ is the unique minimizer of $\hat{J}_i$ in $U_\mathrm{ad}$ it holds that
	\begin{equation}
		\label{eq:bound_on_u}
		\hat{J}_i(u_i^*) \leq \hat{J}_i(u^*) \text{ for all } i \in \mathbb{N}.
	\end{equation}
	Therefore, we have the estimate
	\begin{equation}
		\label{eq:helping_estimate}
		\norm{u_i^*}{L^2(0,T;L^2(\Omega))}^2 \leq C\ \hat{J}_i(u_i^*) \leq C\ \hat{J}_i(u^*) =\  C \left( \frac{1}{2} \norm{G_{c_i}(u^*)- y_d}{L^2(0,T;L^2(\Omega))}^2 + \frac{\lambda}{2} \norm{u^*}{L^2(0,T;L^2(\Omega))}^2 \right).
	\end{equation}
	Using Lemma~\ref{lem:boundedness} and \eqref{eq:helping_estimate}, we get that
	\begin{equation*}
		\norm{u_i^*}{L^2(0,T;L^2(\Omega))}^2 \leq C \left( \norm{u^*}{L^2(0,T;L^2(\Omega))}^2 + \norm{y_d}{L^2(0,T;L^2(\Omega))}^2 + \norm{y_0}{H^1_0(\Omega)}^2 + \norm{y_1}{L^2(\Omega)}^2 \right)
	\end{equation*}
	for a constant $C> 0$ which is independent of $i$. Using this together with Lemma~\ref{lem:boundedness} yields the estimate
	\begin{equation*}
		\norm{y_i^*}{L^2(0,T;H^1_0(\Omega))}^2 + \norm{(y_i^*)'}{L^2(0,T;L^2(\Omega))}^2 \leq C\left( \norm{u^*}{L^2(0,T;L^2(\Omega))}^2 + \norm{y_d}{L^2(0,T;L^2(\Omega))}^2 +\norm{y_0}{H^1_0(\Omega)}^2 + \norm{y_1}{L^2(\Omega)}^2 \right).
	\end{equation*}
	
	Therefore, we have the boundedness of the sequence $(u_i^*)$ in $L^2(0,T;L^2(\Omega))$ as well as the boundedness of both $(y_i^*)$ and $((y_i^*)')$ in $L^2(0,T;H^1_0(\Omega))$ and $L^2(0,T;L^2(\Omega))$, respectively, and can extract subsequences $(u_{i_j}^*)$, $(y_{i_j}^*)$ and $((y_{i_j}^*)')$ such that
	\begin{equation*}
		u_{i_j}^* \rightharpoonup \bar{u} \text{ in } L^2(0,T;L^2(\Omega)), \qquad y_{i_j}^* \rightharpoonup \bar{y} \text{ in } L^2(0,T;H^1_0(\Omega)), \text{ and } \qquad (y_{i_j}^*)' \rightharpoonup \bar{y}' \text{ in } L^2(0,T;L^2(\Omega)).
	\end{equation*}
	As before, $\bar{y}'$ is the weak time derivative of $\bar{y}$. With exactly the same arguments as in Appendix~\ref{proof:convergence_rateless} we observe that $\bar{y}\in W(0,T)$ is the unique weak solution of the heat equation with right-hand side $\bar{u}$, i.e., $\bar{y} = G(\bar{u})$.
	In particular, we get $\hat{J}(\bar{u}) = J(G(\bar{u}), \bar{u}) = J(\bar{y}, \bar{u})$.
	
%	Now, we aim to show that, in fact, $\bar{y} = y^*$ and $\bar{u} = u^*$ holds. 
%	Recall that we have
%	\begin{equation}
%		\label{eq:helping_limit}
%		\int\limits_{0}^{T} \tau_{i_j} \inner{(y_{i_j}^*)''}{v}{} + \left( (y_{i_j}^*)', v\right) + a[y_{i_j}^*,v] \dx{t} = \int\limits_{0}^{T} \left( u_{i_j}^*, v\right) \dx{t}
%	\end{equation}
%	for all $v\in L^2(0,T;H^1_0(\Omega))$ as well as $y_{i_j}^*(0) = y_0$ and $(y_{i_j}^*)'(0) = y_1$ for all $j$ since $y_i^* = G_{c_i}(u_i^*)$. With exactly the same arguments as in Appendix~\ref{proof:convergence_rateless} we observe that \eqref{eq:helping_limit} converges to
%	\begin{equation*}
%	\int\limits_{0}^{T} \left( \bar{y}', v\right) + a[\bar{y},v] \dx{t} = \int\limits_{0}^{T} \left( \bar{u}, v \right) \dx{t}
%	\end{equation*}
%	for all $v\in L^2(0,T;H^1_0(\Omega))$ and that we have the initial condition $\bar{y}(0) = y_0$. Again, this implies that $\bar{y}\in \hat{W}(0,T)$ is the unique weak solution of the heat equation with right-hand side $\bar{u}$, i.e., $\bar{y} = G(\bar{u})$.
%	In particular, we get $\hat{J}(\bar{u}) = J(G(\bar{u}), \bar{u}) = J(\bar{y}, \bar{u})$.
	
	Thanks to Theorem~\ref{thm:convergence_rateless} we get, for a fixed $u$, the strong convergence $G_{c_i}(u) \to G(u)$ in $L^2(0,T;L^2(\Omega))$, so that we also get $\hat{J}_i(u) \to \hat{J}(u)$. The weak lower semi-continuity of $J$ implies the estimate $J(\bar{y},\bar{u}) \leq \liminf_{i\to\infty} J(y_i^*,u_i^*)$.
	
	We use these results for the subsequence $i_j$ and obtain the following
	\begin{equation*}
			\hat{J}(\bar{u}) = J(\bar{y},\bar{u}) \leq \liminf_{j\to\infty} J(y_{i_j}^*,u_{i_j}^*) = \liminf_{j\to\infty} \hat{J}_{i_j}(u_{i_j}^*) \leq \liminf_{j\to\infty} \hat{J}_{i_j}(u^*) = \hat{J}(u^*).
	\end{equation*}
	Recall that $u^*$ is the unique minimizer of $\hat{J}$ in $U_\mathrm{ad}$. Therefore, we have $\bar{u} = u^*$ and, hence, directly $\bar{y} = G(\bar{u}) = G(u^*) = y^*$, which is what we claimed in the beginning. Due to the uniqueness of the weak time derivative we also get $\bar{y}' = (y^*)'$.
	
	Similarly to before, the uniqueness of the limit gives the convergence of the entire sequence, so that we get
	\begin{equation*}
		u_i^* \rightharpoonup u^* \text{ in } L^2(0,T;L^2(\Omega)), \qquad y_i^* \rightharpoonup y^* \text{ in } L^2(0,T;H^1_0(\Omega)), \text{ and } \qquad (y_i^*)' \rightharpoonup (y^*)' \text{ in } L^2(0,T;L^2(\Omega)).
	\end{equation*}
	
	We apply the Aubin-Lions Lemma as we did in Appendix~\ref{proof:convergence_rateless} and deduce the strong convergence
	\begin{equation*}
		y_i^* \to y^* \text{ in } L^2(0,T;L^2(\Omega)).
	\end{equation*}
	To derive the strong convergence of the controls, we need additional estimates. Due to the weak lower semicontinuity of the norm we have
	\begin{equation}
		\label{eq:estimate_chain}
		\norm{u^*}{L^2(0,T;L^2(\Omega))}^2 \leq \liminf_{i\to\infty} \norm{u^*_i}{L^2(0,T;L^2(\Omega))}^2 \leq \limsup_{i\to\infty} \norm{u^*_i}{L^2(0,T;L^2(\Omega))}^2.
	\end{equation}
	Furthermore, we have the estimate
	\begin{equation}
		\label{eq:estimate_control}
		\limsup_{i\to\infty} \hat{J}_i(u_i^*) \leq \limsup_{i\to\infty} \hat{J}_i(u^*) = \hat{J}(u^*),
	\end{equation}
	due to \eqref{eq:bound_on_u}. We rewrite the left-hand side of \eqref{eq:estimate_control} as
	\begin{equation*}
		\limsup_{i\to\infty} \hat{J}_i(u_i^*) = \frac{1}{2} \norm{y^* - y_d}{L^2(0,T;L^2(\Omega))}^2 + \limsup_{i\to\infty} \frac{\lambda}{2} \norm{u_i^*}{L^2(0,T;L^2(\Omega))}^2,
	\end{equation*}
	due to the strong convergence of $y_i^* \to y^*$ in $L^2(0,T;L^2(\Omega))$. Similarly, the right-hand side of \eqref{eq:estimate_control} is given by
	\begin{equation*}
	\hat{J}(u^*) = \frac{1}{2} \norm{y^* - y_d}{L^2(0,T;L^2(\Omega))}^2 + \frac{\lambda}{2} \norm{u^*}{L^2(0,T;L^2(\Omega))}^2.
	\end{equation*}
	Hence, \eqref{eq:estimate_control} reveals that
	\begin{equation*}
		\limsup_{i\to\infty} \norm{u_i^*}{L^2(0,T;L^2(\Omega))}^2 \leq \norm{u^*}{L^2(0,T;L^2(\Omega))}^2,
	\end{equation*}
	since $\lambda >0$. With this available, we extend the estimate \eqref{eq:estimate_chain} to
	\begin{equation*}
		\norm{u^*}{L^2(0,T;L^2(\Omega))}^2 \leq \liminf_{i\to\infty} \norm{u^*_i}{L^2(0,T;L^2(\Omega))}^2 \leq \limsup_{i\to\infty} \norm{u^*_i}{L^2(0,T;L^2(\Omega))}^2 \leq \norm{u^*}{L^2(0,T;L^2(\Omega))}^2.
	\end{equation*}
	As a result, we have that
	\begin{equation}
		\label{eq:norm_convergence}
		\lim\limits_{i\to\infty} \norm{u_i^*}{L^2(0,T;L^2(\Omega))}^2 = \norm{u^*}{L^2(0,T;L^2(\Omega))}^2.
	\end{equation}
	Hence, the weak convergence of the sequence $u_i^*$ together with \eqref{eq:norm_convergence} gives the strong convergence $u_i^* \to u^*$ in $L^2(0,T;L^2(\Omega))$ due to the Radon-Riesz theorem (cf.~\cite[Lemma~E8.5]{Alt2016Linear}), which completes the proof.
\end{proof}

\section*{References}
\bibliography{lit}

\end{document}